\documentclass{article}

\usepackage{color}
\usepackage{amsmath,amssymb}
\usepackage{amscd}
\usepackage{mathrsfs}
\usepackage{amsthm}
\theoremstyle{plain}

\newtheorem{theorem}{\bf Theorem}[section]
\newtheorem{lemma}[theorem]{\bf Lemma}
\newtheorem{proposition}[theorem]{\bf Proposition}
\newtheorem{corollary}[theorem]{\bf Corollary}

\theoremstyle{definition}
\newtheorem{definition}[theorem]{\bf Definition}
\newtheorem{example}[theorem]{\bf Example}
\newtheorem{remark}[theorem]{\bf Remark}

\newcommand{\nai}[2]{\langle #1,#2\rangle}
\newcommand{\mattwo}[4]{\begin{pmatrix}#1 & #2 \\ #3 & #4\end{pmatrix}}
\newcommand{\eqa}[1]{
\begin{align*}
#1
\end{align*}}

\usepackage[all]{xy}

\title{Polish groups of unitaries}
\author{Hiroshi Ando \and Yasumichi Matsuzawa} 
\begin{document}
\maketitle
\begin{abstract}
We study the question of which Polish groups can be realized as subgroups of the unitary group of a separable infinite-dimensional Hilbert space. We also show that for a separable unital C$^*$-algebra $A$, the identity component $\mathcal{U}_0(A)$ of its unitary group has property (OB) of Rosendal (hence it also has property (FH)) if and only if the algebra has finite exponential length (e.g. if it has real rank zero), while in many cases the unitary group $\mathcal{U}(A)$ does not have property (T). On the other hand, the $p$-unitary group $\mathcal{U}_p(M,\tau)$ where $M$ is a properly infinite semifinite von Neumann algbera with separable predual, does not have property (FH) for any $1\le p<\infty$. This in particular solves a problem left unanswered in the work of Pestov \cite{Pestov18}.  
\end{abstract}


\medskip


\medskip
\section{Introduction and Main Results}
Unitary representation theory of topological groups has a long history and has applications to diverse fields including mathematical physics, ergodic theory and operator algebra theory. For locally compact groups, the existence of the Haar measure guarantees the existence of sufficiently many unitary representations which enables us to study fine structures of these groups. The class of Polish groups is a natural and rich family of topological groups which include all second countable locally compact groups and many others. It is of interest to study which Polish groups admit abundance of unitary representations and in particular, which groups arise as groups of unitaries on some Hilbert space. It is known that there exist exotic Polish groups, namely those which do not have nontrivial strongly continuous unitary representations at all. The first such example is given by Herer--Christensen \cite{HC75}. See also the work of Megrelishvili \cite{Megrelishvili08} and Carderi--Thom \cite{CT18} for more exotic groups with surprising properties. For abelian exotic groups, see also Banaszczyk's book \cite{Banaszczyk}. If a Polish group in question does admit a realization as a group of unitaries, there are more structural questions one may ask. The question of amenability or the property (T) of Kazhdan \cite{Kazhdan67} are such examples (see $\S$\ref{def: various properties} for the definition of these and related properties of topological groups). 
Here, a (not necessarily locally compact) topological group $G$ is said to be amenable if there exists a left-invariant mean on the space ${\rm{LUCB}}(G)$ of bounded left uniformly continuous functions on $G$. This is equivalent to the standard definition of the amenability if $G$ is locally compact.  
The situation becomes drastically different for non-locally compact groups. For example, it is well-known that a locally compact group which is amenable and has property (T) is necessarily compact. On the other hand, the unitary group $\mathcal{U}(\ell^2)$ equipped with the strong operator topology is amenable by de la Harpe's result \cite{delaHarpe73} (see also \cite{delaHarpe78}) and at the same time has property (T) by Bekka's result \cite{Bekka03}, although it is far from being compact. 
In fact, $\mathcal{U}(\ell^2)$ is known to be extremely amenable by Gromov--Milman's work \cite{GromovMilman}, which means that whenever it acts continuously on a compact Hausdorff space, there is a fixed point. This is another remarkable phenomenon which cannot occur in the locally compact setting. 
As another example, the equivalence of the property (FH) and property (T), which hold for $\sigma$-compact locally compact groups, fails for non-locally compact Polish groups (the implication (T)$\Rightarrow$(FH) is valid for any topological group by Delorme's work \cite{Delorme77}, and the implication (FH)$\Rightarrow$(T) in the $\sigma$-compact locally compact setting is due to Guichardet \cite{Guichardet72}). See Bekka--de la Harpe--Valette' book \cite[$\S$2.12]{BHV} for more details). A recent intensive work by Pestov \cite{Pestov18} shows that nevertheless, the amenability and the property (T) are in some sense still opposite properties to each other for SIN Polish groups. For example, if a norm-closed subgroup of $\mathcal{U}(\ell^2)$ is amenable and has property (T), then it is maximally almost periodic \cite[Theorem 1.4]{Pestov18}, meaning that finite-dimensional strongly continuous unitary representations separate points. 
On the other hand, Rosendal \cite{Rosendal09,Rosendal13,Rosendal18} has clarified that it is profitable to study large scale geometric structures of non-locally compact Polish groups. One of the central concepts in his study of the coarse geometry of topological groups is the notion of coarsely boundedness (it is called the property (OB) in Rosendal \cite{Rosendal09}), which is related to the boundedness in the sense of Hejcman \cite{Hejcman58}. Here, a Polish group has property (OB) if whenever it acts continuously by isometries on a metric space, all orbits are bounded. The properties (OB), (FH) and (T) are thus closely related, and the distinction of them provides better understanding of non-locally compact Polish groups from the large scale viewpoint. 
The purpose of this paper is to study some of the above mentioned themes, focusing on the two different topologies on the unitary groups, namely the norm topology and the strong operator topology. We also take advantage of operator algebra theory for the study of Polish groups of unitaries. In order to explain the results of this article, let us introduce some terminology. 
We say that a Hausdorff topological group $G$ is 
\begin{list}{}{}
\item[(i)] {\it strongly unitarily representable} (SUR) if $G$ is isomorphic as a topological group to a closed subgroup of $\mathcal{U}(H)$ for some Hilbert space $H$ equipped with the strong operator topology. 
\item[(ii)] {\it norm unitarily representable} (NUR) if $G$ is isomorphic as a topological group to a closed subgroup of $\mathcal{U}(H)$ for some Hilbert space $H$ equipped with the norm topology. 
\end{list} 

In $\S$\ref{sec: UR} we study the unitary representability of some Polish groups consisting of unitary operators. Our work in this section is inspired by the works of Megrelishvili \cite{Megrelishvili08} and Galindo \cite{Galindo09} about SUR. If $G$ is a subgroup of $\mathcal{U}(\ell^2)$, we denote by $G_u$ (resp. $G_s$) the group $G$ equipped with the norm (resp. the strong) topology.  
It is obvious that if $M$ is a von Neumann algebra, then the unitary group $\mathcal{U}(M)_s$ is SUR. 
On the other hand, the group $\mathcal{U}(M)_s$ is NUR if and only if $M$ is finite-dimensional (Corollary \ref{cor: U(M)_s not NUR}). 
It is also clear that if $A$ is a separable unital C$^*$-algebra, then $\mathcal{U}(A)_u$ is NUR. On the other hand, $\mathcal{U}(A)_u$ is SUR if and only if $A$ is finite-dimensional (Theorem \ref{thm: UR iff finite-dim}). These results show the incompatibility of the two topologies for the non-locally compact case. Although the result is to be expected from the fact that the only von Neumann algebras which are separable in norm are finite-dimensional ones, the proof is entirely different. Theorem \ref{thm: UR iff finite-dim} is used to prove that the Fredholm unitary group $\mathcal{U}_{\infty}(\ell^2)$ (Definition \ref{def: p-unitary group}) with the norm topology, which is evidently NUR, is not SUR. 
More interesting example of Polish groups of unitaries are the so-called Schatten $p$-unitary groups $\mathcal{U}_p(\ell^2)\,(1\le p<\infty)$ (Definition \ref{def: p-unitary group}). 
Although there have been several intensive works on the (projective) unitary representations of $\mathcal{U}_p(H)$ for $p=1,2$, not much is known about unitary representability of $\mathcal{U}_p(\ell^2).$ 
We know that $\mathcal{U}_2(\ell^2)$ is SUR (see \cite[Theorem 3.2]{AM12-2}; in fact it is of finite type, meaning that it embeds into the strongly closed subgroup of a II$_1$ factor with separable predual), but it is unclear whether $\mathcal{U}_2(\ell^2)$ is NUR. 
Also it is  unknown whether $\mathcal{U}_p(\ell^2)$ is SUR or NUR for $p\neq 1,2,\infty$ (if it is SUR, then by \cite[Theorem 2.22]{AM12-2}, it is actually of finite type). In contrast, we show in $\S$\ref{subsec: NUR} that $\mathcal{U}_1(\ell^2)$ is NUR (Theorem \ref{U1 is NUR}). We do not know if it is SUR. This might suggest that the family $\mathcal{U}_p(\ell^2)$ shares mixed features of both the NUR groups and the SUR groups. 

In $\S$\ref{sec: boundedness}, we study whether the following properties hold for a specific class of Polish groups of unitaries: boundedness, properties (FH), (OB) and (T). If $A$ is a unital abelian C$^*$-algebra, we show that $\mathcal{U}(A)_u$ is bounded if and only if its spectrum is totally disconnected, which is known to be equivalent to $A$ having real rank zero.  For general $A$, $\mathcal{U}(A)_u$ is bounded if and only if the identity component $\mathcal{U}_0(A)_{u}$ is bounded and has finite index in $\mathcal{U}(A)_u$. Thus we focus on $\mathcal{U}_0(A)_{u}$. We show that $\mathcal{U}_0(A)_{u}$ is bounded if and only if it has finite exponential length (Definition \ref{def: exp length}) introduced by Ringrose \cite{Ringrose91} (Theorem \ref{thm: bounded and cel}). Phillips' survey \cite{Phillips93} is a very readable account of this and the related notion of the exponential rank. The notion of exponential length has been studied by experts, and it was one of the important ingredients in Phillips' proof \cite{Phillips00} of the celebrated Kirchberg--Phillips' classification Theorem for purely infinite simple nuclear C$^*$-algebras in the UCT class. 
Thanks to Lin's result \cite{Lin93} that a C$^*$-algebra $A$ having real rank zero has property weak (FU) of Phillips \cite{Phillips91}, meaning that any $u\in \mathcal{U}_0(A)$ is the norm limit of unitaries with finite spectrum, it follows that such an algebra has finite exponential length. But there exist unital C$^*$-algebras with finite exponential lengths which do not have real rank zero (see \cite{Phillips92} and Remark \ref{rem: real rank zero is not necessary}). Thus, the real rank zero property is not necessary for $\mathcal{U}_0(A)_{u}$ to be bounded. 
In \cite[Example 6.7]{Pestov18}, Pestov shows that the $\mathcal{U}_2(\ell^2)$ does not have property (T), and it was left unanswered (see \cite[$\S$3.5]{Pestov18}) whether $\mathcal{U}_2(\ell^2)$ has property (FH). We show that it does not. Actually, we prove a more general result that if $M$ is a properly infinite semifinite von Neumann algebra with a faithful normal semifinite trace $\tau$, then the $L^p$-unitary group $\mathcal{U}_p(M,\tau)$ (see Definition \ref{def: p-unitary group}) does not have property (FH) for $1\le p<\infty$. In fact, we prove that any closed subgroup of $\mathcal{U}_p(M,\tau)$ which has property (FH) is in fact of finite type (Theorem \ref{not FH new}, and Theorem \ref{thm: (FH) implies Ufin for p>2}). In particular, any closed subgroup of $\mathcal{U}_p(\ell^2)\,(1\le p<\infty)$ with property (FH) must be compact (Corollary \ref{cor: closed subgroup of U_pl2 is compact} and Corollary \ref{cor: closed subgroup of U_pl2 is compact p>2}). For technical reasons, we need two different proofs depending on whether $1\le p\le 2$ or $2<p<\infty$. This comes from the fact that only for $1\le p\le 2$, there is a well-defined affine isometric action of $\mathcal{U}_p(M,\tau)$ on $L^2(M,\tau)$ (regarded as a real Hilbert space by taking the real part of the inner product) given by 
\[u\cdot x=ux+(u-1),\,\,u\in \mathcal{U}_p(M,\tau),\,x\in L^2(M,\tau).\]
It might be interesting to note that the absence of property (FH) for $\mathcal{U}_p(M,\tau)$ is valid even for $M$ of the form $N\overline{\otimes}\mathbb{B}(\ell^2)$, where $N$ is a type II$_1$ factor with property (T) in the sense of \cite{ConnesJones85}. 
Finally, we show that if $A$ is a unital separable C$^*$-algebra admitting a representation $\pi$ which generates a type II$_1$ factor with separable predual having property Gamma, then $\mathcal{U}(A)_u$ does not have property (T) (Theorem \ref{thm: Gamma negates (T)}). This is an immediate consequence of the definition itself, but shows that there exist many C$^*$-algebras whose unitary groups do not have property (T), while at the same time many of them do have property (FH) (or even (OB)). In particular, the unitary group of the UHF algebra $M_{2^{\infty}}$ of type $2^{\infty}$ does not have property (T) but has property (OB).   
\section{Preliminaries}\label{sec: preliminaries}
We consider complex Hilbert spaces unless explicitly stated otherwise. For a Hilbert space $H$, we denote  by $\mathcal{U}(H)$ (resp. $\mathbb{K}(H)$) the multiplicative  (resp. additive) group of unitary operators (resp. compact operators) of $H$. $\sigma(a)$ denotes the spectrum of a closed, densely defined operator $a$ on $H$. For $1\le p<\infty$, $S_p(H)$ is the additive group of Schatten $p$-class operators on $H$. We use the following abbreviations for the operator topologies: SOT for the strong operator topology, SRT for the strong resolvent topology (see e.g., \cite[$\S$2]{AM15}). We sometimes use $\|\cdot\|_{\infty}$ to denote the operator norm, if it is appropriate to distinguish it from other norms, such as the Schatten $p$-norm $\|\cdot\|_p$. 
For a topological group $G$, we denote by $\mathscr{N}(G)$ the set of all open  neighborhoods of the identity in $G$. 
If $G$ is a subgroup of $\mathcal{U}(H)$, we often denote by $G_u$ or $(G,\|\cdot\|_{\infty})$ (resp. $G_s$ or $(G,{\rm{SOT}})$) the group $G$ equipped with the norm (resp. the strong operator) topology. For a unital C$^*$-algebra $A$, its unitary group (resp. the set of self-adjoint elements in $A$) is denoted by $\mathcal{U}(A)$ (resp. $A_{\rm{sa}}$). $\mathcal{U}_0(A)$ is the identity component of $\mathcal{U}(A)$ in the norm topology.  

\begin{definition}\label{def: Lp(M)}
Let $M$ be a semifinite von Neuman algebra with separable predual and $\tau$ a normal faithful semifinite trace on $M$. 
The noncommutative $L^p$-space associated with $M$ is denoted by $L^p(M,\tau)$. We denote by $\|\cdot\|_p$ the $p$-norm induced by $\tau$. If $x$ is a closed, densely defined operator affiliated with $M$ and $|x|=\int_0^{\infty}\lambda\,{\rm{d}}e(\lambda)$ is the spectral resolution of $|x|$, then 
\[\|x\|_p=\left (\int_{[0,\infty)}\lambda^p\,{\rm{d}}\tau(e(\lambda))\right )^{\frac{1}{p}}.\]
Details about non-commutative integration theory can be found e.g., in \cite[Chapter IX.2]{Takesakibook}. 
\end{definition}
\begin{definition}[\cite{AM12-2}]\label{def: p-unitary group} Let $(M,\tau)$ be as in Definition \ref{def: Lp(M)} and $1\le p<\infty$. The {\it $p$-unitary group} of $M$, denoted by $\mathcal{U}_p(M,\tau)$ is the group of all unitaries $u\in \mathcal{U}(M)$ such that $u-1\in L^p(M,\tau)$. The group is endowed with the metric topology given by the following bi-invariant metric 
\[d(u,v)=\|u-v\|_{p},\ \ u,v\in \mathcal{U}_p(M).\]
In the case $M=\mathbb{B}(H)$ for a Hilbert space $H$ and $\tau$ is the usual trace, the group is called the {\it Shcatten $p$-unitary group}  and is denoted by $\mathcal{U}_p(H)$. The group $\mathcal{U}_{\infty}(H)=\{u\in \mathcal{U}(\ell^2)\mid u-1\in \mathbb{K}(H)\}$ endowed with the norm topology is sometimes called the {\it Fredholm unitary group}. If $M$ is finite, then $\mathcal{U}_p(M,\tau)=\mathcal{U}(M)_s$ is independent of the choice of $\tau$ and $p$. 
\end{definition}
\begin{definition}\label{def: various properties}
Let $G$ be a topological group, $\pi\colon G\to \mathcal{U}(H)$ be a strongly continuous unitary representation.
\begin{list}{}{}
\item[{\rm{(i)}}] For $(\emptyset\neq)Q\subset G$ and $\varepsilon>0$, a vector $\xi\in H$ is {\it $(Q,\varepsilon)$-invariant} if the following inequality holds:
\[\sup_{x\in Q}\|\pi(x)\xi-\xi\|<\varepsilon \|\xi\|.\]
\item[{\rm{(ii)}}] We say that $\pi$ {\it almost has invariant vectors}, if it has a $(Q,\varepsilon)$-invariant vector for every compact subset $Q$ of $G$ and every $\varepsilon>0$.\\
We say that   
\item[{\rm{(iii)}}] $G$ has {\it Property }(T) if every strongly continuous unitary representation of $G$ which almost has invariant vectors admits a nonzero invariant vector. 
\item[(iv)] $G$ has {\it Property} (FH) if every continuous affine isometric action of $G$ on a real Hilbert space has a fixed point. 
\item[(v)] $G$ has {\it Property} (OB) if whenever $G$ acts continuously on a metric space by isometries, all orbits are bounded. 
\item[(vi)] $G$ is {\it bounded} if for every $V\in \mathscr{N}(G)$, there exists $n\in \mathbb{N}$ and a finite set $F\subset G$ such that $G=V^nF$. 
\end{list}
\end{definition}
The following results help us understand the implications among the above properties. 

\begin{theorem}[{\cite[Theorem 1.11, Proposition 1.14]{Rosendal13}}]
Let $G$ be a Polish group. Then 
\begin{list}{}{}
\item[{\rm{(i)}}] The following three conditions are equivalent.
\begin{list}{}{}
\item[{\rm{(a)}}] $G$ has property {\rm{(OB)}}. 
\item[{\rm{(b)}}] For every symmetric $V\in \mathscr{N}(G)$, there exist a finite subset $F\subset G$ and $k\in \mathbb{N}$ such that $G=(FV)^k$, .
\item[{\rm{(c)}}] Every compatible left-invariant metric on $G$ is bounded. 
\end{list}
\item[{\rm{(ii)}}] The following three conditions are equivalent. 
\begin{list}{}{}
\item[{\rm{(a)}}] $G$ is bounded. 
\item[{\rm{(b)}}] Every left uniformly continuous function on $G$ is bounded.
\item[{\rm{(c)}}] $G$ has property {\rm{(OB)}} and every open subgroup is of finite index.  
\end{list}
\end{list}
\end{theorem}
Here, a map $f$ from a topological group $G$ to a metric space $(X,d)$ is called left uniformly continuous, if for every $\varepsilon>0$, there exists $U\in \mathscr{N}(G)$ such that $\sup_{x\in G}d(f(ux),f(x))<\varepsilon$ holds (this is sometimes called the right uniform continuity in the literature). 
Also, a topological group $G$ has property (FH) if and only if for every continuous affine isometric action of $G$ on a real Hilbert space, there is a bounded orbit (see e.g., \cite[Proposition 2.2.9]{BHV}). Therefore, the following implications hold: 
\[\text{bounded}\Rightarrow\text{(OB)}\Rightarrow\text{(FH)}\]
It is known that none of the above two implications can be reversed in general. 
Also, (T)$\Rightarrow$(FH)
hold for any topological group \cite{Delorme77} (see also \cite[Theorem 2.12.4]{BHV}). In general, there is no implication between (T) and (OB). As pointed out in \cite{Pestov18}, any countable discrete group with property (T) does not have property (OB) (consider the isometric action on its Cayley graph). On the other hand, $\mathcal{U}_{\infty}(\ell^2)$ is bounded \cite{Atkin89} hence has property (FH), but it fails to have property (T) \cite[Example 6.6]{Pestov18} (see also Remark \ref{rem: not (T)}).  

For more details about the above mentioned properties, we refer the reader to \cite{BHV} for property (T) and (FH), \cite{Rosendal09,Rosendal13} for property (OB) and boundedness (see also \cite{Hejcman58}).

\section{Unitary Representability}\label{sec: UR}
In this section we study a class of Polish groups which arise as closed subgroups of the unitary group. 
\begin{definition}
We say that a Hausdorff topological group $G$ is 
\begin{list}{}{}
\item[(i)] {\it strongly unitarily representable} (SUR) if $G$ is isomorphic as a topological group to a closed subgroup of $\mathcal{U}(H)$ for some Hilbert space $H$ equipped with the strong operator topology. 
\item[(ii)] {\it norm unitarily representable} (NUR) if $G$ is isomorphic as a topological group to a closed subgroup of $\mathcal{U}(H)$ for some Hilbert space $H$ equipped with the norm topology. 
\item[(iii)] of {\it finite type}, if $G$ is isomorphic as a topological group to a closed subgroup of $\mathcal{U}(M)_s$ for a finite type von Neumann algebra $M$.
\item[(iv)] {\it SIN} (Small Invariant Neighborhoods) if for every $U\in \mathscr{N}(G)$, there exists $V\in \mathscr{N}(G)$ such that $V\subset U$ and $gVg^{-1}=V$ for every $g\in G$. 
\end{list} 
\end{definition}
If $G$ is a metrizable topological group, then $G$ is SIN, if and only if $G$ admits a compatible two-sided invariant metric. Thus, every NUR group is SIN (the metric is given by the norm). It is also clear that a finite type Polish group is SIN and SUR. The converse need not be true \cite{AMTT}. However, we do not know whether there exists a connected SUR SIN Polish group which is not of finite type. 

Recall also that a topological group $G$ {\it has small subgroups} if for every neighborhood $U$ of the identity in $G$, there exists a non-trivial subgroup of $G$ contained in it. 
\begin{remark}
Every locally compact Polish groups is SUR via the left regular representation (and it is also NUR if the group is discrete). The solution to the Hilbert's 5th problem (see e.g., \cite{Taobook}) asserts that a compact Polish group does not have small subgroups if and only if it is a compact Lie group, in which case the group can be realized as a closed subgroup of $\mathcal{U}(n)$ for some $n\in \mathbb{N}$, which is NUR. Thus for compact Polish groups, NUR is equivalent to the absence of small subgroups. If a locally compact Polish group is totally disconnected but perfect, then by van Dantzig's Theorem (see. e.g., \cite{Taobook}), it has small subgroups and therefore is not NUR (see Proposition \ref{prop: small subgroups} below). 
\end{remark}

\subsection{Strong Unitary Representability (SUR)}\label{subsec: SUR}
We start by showing the next theorem. 
\begin{theorem}\label{thm: UR iff finite-dim}
Let $A$ be a unital ${\rm{C}}^*$-algebra. Then the following conditions are equivalent. 
\begin{list}{}{}
\item[{\rm{(i)}}] The additive group $A$ with the norm topology is {\rm{SUR}}. 
\item[{\rm{(ii)}}] The additive group $A_{\rm{sa}}$ with the norm topology is {\rm{SUR}}.
\item[{\rm{(iii)}}] $\mathcal{U}(A)_u$ is {\rm{SUR}}. 
\item[{\rm{(iv)}}] $A$ is finite-dimensional. 
\end{list}
\end{theorem}
The proof is inspired by the works of Megrelishvili \cite{Megrelishvili08} and Galindo \cite{Galindo09}. Recall the notion of cotype. We refer the reader to \cite{AlbiacKalton06} for details.  
\begin{definition}
A Banach space $X$ has {\it cotype} $q\in [2,\infty)$ if there exists a constant $C_q>0$ such that for every finite set of vectors $x_1,\dots,x_n\in X$ the following inequality holds: 
\begin{equation}
\left (\sum_{i=1}^n\|x_i\|^q\right )^{\frac{1}{q}}\le C_q\,\mathbb{E}\left \|\sum_{i=1}^n\varepsilon_ix_i\right \|^{\frac{1}{q}},
\end{equation}
where $(\varepsilon_n)_{n=1}^{\infty}$ is a Rademacher sequence, i.e., a sequence of i.i.d.'s on a probability space with $\mathbb{P}(\varepsilon_n=1)=\mathbb{P}(\varepsilon_n=-1)=\tfrac{1}{2}$ for all $n\in \mathbb{N}$. If $X$ has cotype $q$ for some $q\in [2,\infty)$, we say that $X$ has {\it nontrivial cotype}.  
\end{definition}
We will also use the following characterization of the SUR, which has been known to experts (we follow \cite[Theorem 2.1]{Gao05}). 
\begin{theorem}[Folklore]\label{thm: SUR positive type} Let $G$ be a Polish group with the  identity $1_G$. The following conditions are equivalent. 
\begin{list}{}{}
\item[{\rm{(i)}}] $G$ is {\rm{SUR}}. 
\item[{\rm{(ii)}}] $G$ is isomorphic as a topological group to a closed subgroup of $\mathcal{U}(\ell^2)_s$.
\item[{\rm{(iii)}}] There exists a family $\mathcal{F}$ of continuous positive definite functions on $G$ which generates a neighborhood basis of $1_G$ in the sense that for every $V\in \mathscr{N}(G)$, there exist $f_1,\dots, f_n\in \mathcal{F}$ and open sets $O_1,\dots,O_n$ in $\mathbb{C}$ such that 
\[1_G\in \bigcap_{i=1}^nf_i^{-1}(O_i)\subset V.\]
\item[{\rm{(iv)}}] There exists a family $\mathcal{F}$ of continuous positive definite functions on $G$ which separates $1_G$ and closed sets not containing $1_G$ in the sense that whenever $A$ is a closed set in $G$ with $1_G\notin A$, there exists an $f\in \mathcal{F}$ such that the following inequality holds:
\[\sup_{g\in A}|f(g)|<f(1_G).\] 
\end{list}
\end{theorem}
We prepare several lemmas. 
\begin{lemma}\label{lem: URcotype2} If a Banach space embeds as a topological vector space into $L^0(\Omega,\mu)$ equipped with the measure topology for some probability space $(\Omega,\mu)$, then it has cotype 2. 
\end{lemma}
\begin{proof} See \cite[Corollary 8.17]{BenyaminiLindenstrauss00}. 
\end{proof}
\begin{lemma}\label{lem: UR then embeds to L0} Let $X$ be a separable real Banach space. If $X$ as an additive topological group is {\rm{SUR}}, then there exists a probability space $(\Omega,\mu)$ such that $X$ embeds isomorphically into $L^0(\Omega,\mu)$ equipped with the topology of convergence in measure. 
\end{lemma}
\begin{proof}
Let $\pi$ be a topological group isomorhism of $X$ onto an SOT-closed subgroup of $\mathcal{U}(\ell^2)$. The image $\pi(X)$ generates an abelian von Neumann algebra $A$ with separable predual. Then we may identify $A$ with $L^{\infty}(\Omega,\mu)$ where $(\Omega,\mu)$ is a probability space. For each $a\in X$, the map $\mathbb{R}\ni t\mapsto \pi(ta)\in \mathcal{U}(A)$ is a strongly continuous one-parameter unitary group. By Stone Theorem, there exists a unique self-adjoint operator $T(a)$ affiliated with $A$ such that $e^{itT(a)}=\pi(ta)\ (t\in \mathbb{R})$. Since $X$ is an abelian group, it is easy to see that $T\colon X\to L^0(\Omega,\mu)$ is a real linear map. 
We show that $T$ is a topological group isomorphism from $X$ onto its image. Since $X,L^0(\Omega,\mu)$ are both metrizable, it suffices to show that for a sequence $(a_n)_{n=1}^{\infty}$ in $X$, $\|a_n\|\stackrel{n\to \infty}{\to}0$ if and only if $T(a_n)\stackrel{n\to \infty}{\to}T(0)=0$ in measure. Now we use \cite[Theorem VIII.21]{ReedSimonI} and the fact that the convergence in measure in $L^0(\Omega,\mu)$ is the same as the strong resolvent convergence by \cite[Lemma 3.12]{AM12-1}. Then 
\eqa{
\|a_n\|\stackrel{n\to \infty}{\to}0&\Leftrightarrow \forall t\in \mathbb{R}\ [\pi(ta_n)=e^{itT(a_n)}\stackrel{n\to \infty}{\to}1\ (\text{SOT})]\\
&\Leftrightarrow T(a_n)\stackrel{n\to \infty}{\to}0\ (\text{SRT})\\
&\Leftrightarrow T(a_n)\stackrel{n\to \infty}{\to}0\ (\text{in measure}).
}  
This finishes the proof, because $T(X)$ is then a Polish subgroup of the Polish group $L^0(\Omega,\mu)$, which therefore must be closed in $L^0(\Omega,\mu)$. 
\end{proof}
\begin{lemma}\label{lem: NRT implies norm}
Let $(a_n)_{n=1}^{\infty}$ be a sequence of bounded self-adjoint operators on $H$ converging to $0$ in the norm resolvent topology. Then $\lim_{n\to \infty}\|a_n\|=0$. 
\end{lemma}
\begin{proof}Let $a_n=\int_{\sigma(a)}\lambda\,{\rm{d}}E_{a_n}(\lambda)$ be the spectral resolution of $a_n\,(n\in \mathbb{N})$. First, we show that $\sup_{n\in \mathbb{N}}\|a_n\|<\infty$. Assume by contradiction that $(a_n)_{n=1}^{\infty}$ is unbounded. Then for each $k\in \mathbb{N}$, there exists $n_k\in \mathbb{N}$ with $n_1<n_2<\cdots <n_k$, such that 
$E_{a_{n_k}}(\mathbb{R}\setminus [-k,k])\neq 0$. Let $\xi_k$ be a unit vector contained in $\text{ran}\,E_{a_{n_k}}(\mathbb{R}\setminus [-k,k])$. Then 
\eqa{
\|(a_{n_k}-i)^{-1}\xi_k-(0-i)^{-1}\xi_k\|^2
&=\int_{|\lambda|>k}\frac{\lambda^2}{\lambda^2+1}\,{\rm{d}}\|E_{a_{n_k}}(\lambda)\xi_k\|^2\\
&\ge \tfrac{k^2}{k^2+1}\stackrel{k\to \infty}{\not \to}0,
}
which contradicts the assumption. Therefore $C=\sup_{n\in \mathbb{N}}\|a_n\|<\infty$, and  
\eqa{
\|a_n\|=\|(a_n-i)[(0-i)^{-1}-(a_n-i)^{-1}](0-i)\|\le (C+1)\|(a_{n}-i)^{-1}-(0-i)^{-1}\|\stackrel{n\to \infty}{\to}0.
}
\end{proof}

\begin{proposition}\label{prop: key prop U(A) UR A UR}
Let $A$ be a separable unital abelian {\rm{C}}$^*$-algebra such that $\mathcal{U}(A)_u$ is {\rm{SUR}}. Then the additive group $A_{\rm{sa}}$ with the norm topology is {\rm{SUR}}. 
\end{proposition}
\begin{proof}
Since $\mathcal{U}(A)_u$ is an SUR Polish group, there exists a topological group isomorphism $\pi$ of $\mathcal{U}(A)_u$ onto an SOT-closed subgroup of $\mathcal{U}(H)$ for some separable Hilbert space $H$. 
Define for each $n\in \mathbb{N}$ a continuous function $\psi_{n}\colon A_{\rm{sa}}\to \mathbb{C}$ by 
\[\psi_n(a)=\int_0^{\infty}e^{-t}\nai{\pi(e^{-iat})\xi_n}{\xi_n}\,{\rm{d}}t,\ \ \ a\in A,\]
where $\{\xi_n\}_{n=1}^{\infty}$ is a dense subset of the unit ball of $H$. 
Since $A$ is abelian, the map $A_{\rm{sa}}\ni a\mapsto \nai{\pi(e^{-iat})\xi_n}{\xi_n}\in \mathbb{C}$ is a continuous positive definite function for all $t\in \mathbb{R}$. This implies that $\psi_n$ is a continuous positive definite function as well. By Stone Theorem, for each $a\in A_{\rm{sa}}$ there exists a self-adjoint operator $T(a)$ on $H$ such that $\pi(e^{ita})=e^{itT(a)}\ (t\in \mathbb{R})$ holds.  
Note that $\psi_n(a)$ is then expressed as 
\begin{equation}
\psi_n(a)=-i\nai{(T(a)-i)^{-1}\xi_n}{\xi_n}.
\end{equation}
We show that the family $\{\psi_n\}_{n=1}^{\infty}$ generates a neighborhood basis of 0 in $A_{\rm{sa}}$. Let $\tau$ be the norm topology on $A_{\rm{sa}}$ and $\tau'$ be the weakest topology on $A_{\rm{sa}}$ making the functions $\psi_n\,(n\in \mathbb{N})$ continuous. Then $\tau'$ has a countable basis 
$$\mathcal{B}=\left \{\bigcap_{j=1}^J\psi_{n_j}^{-1}(U_{m_j})\,\middle|\, J\in \mathbb{N},\, n_j,m_j\in \mathbb{N}\,(1\le j\le J)\right\},$$
where $\{U_m\}_{m=1}^{\infty}$ is a countable basis for $\mathbb{R}$. Thus both $\tau$ and $\tau'$ are second countable, and by definition $\tau'\subset \tau$ holds. Note that because $A$ is abelian, it holds that for a net $(a_j)_{j\in J}$ in $A_{\rm{sa}}$ and $a\in A_{\rm{sa}}$, 
\eqa{
a_j\stackrel{j\to \infty}{\to}a\ (\text{in\,}\tau')
&\Leftrightarrow \forall n\in \mathbb{N}\ \psi_n(a_j)\stackrel{j\to \infty}{\to}\psi_n(a)\\
&\Leftrightarrow \forall n\in \mathbb{N}\ \nai{(T(a_j)-i)^{-1}\xi_n}{\xi_n}\stackrel{j\to \infty}{\to}\nai{(a-i)^{-1}\xi_n}{\xi_n}\\
&\stackrel{(*)}{\Leftrightarrow} \forall \xi,\eta \in \mathbb{N}\ \nai{(T(a_j)-i)^{-1}\xi}{\eta}\stackrel{j\to \infty}{\to}\nai{(a-i)^{-1}\xi}{\eta}\\
&\stackrel{(**)}{\Leftrightarrow} T(a_j)\stackrel{\rm{SRT}}{\to}T(a)
}
Here, we used in $(*)$ the density of $\{\xi_n\}_{n=1}^{\infty}$ and the polarization identity, in $(**)$ the fact that  the weak resolvent convergence implies the strong resolvent convergence for self-adjoint operators.
Therefore, thanks to the fact that SRT is a vector space topology on the set $T(A_{\rm{sa}})$ of strongly commuting self-adjoint operators \cite[Lemma 3.12]{AM12-1}, it follows that $\tau'$ is a vector space topology on $A_{\rm{sa}}$. 
We then show that $\tau'\supset \tau$, thus $\tau'=\tau$. By the first countability of both $\tau$ and $\tau'$, it suffices to show that the notion of sequential convergence to 0 in both topologies agree. Let $(a_n)_{n=1}^{\infty}$ be a sequence in $A_{\rm{sa}}$ converging to $0$ in $\tau'$. Then $T(a_n)\stackrel{n\to \infty}{\to}T(0)$ (SRT), whence by \cite[Theorem VIII.21]{ReedSimonI}, 
we have $\forall t\in \mathbb{R}\ [e^{itT(a)}\stackrel{n\to \infty}{\to}e^{itT(0)}\ (\text{SOT})]$, so that because $\pi$ is a topological embedding, the last condition is equivalent to 
$\forall t\in \mathbb{R}\ [\|e^{ita_n}-1\|\stackrel{n\to \infty}{\to}0]$, which implies that $(a_n-0)^{-1}\stackrel{n\to \infty}{\to}(0-i)^{-1}$ in norm. Indeed, because each $a_n$ is bounded, the map $t\mapsto e^{ita_n}$ is norm-continuous. Thanks to the dominated convergence theorem, it then follows that 
\eqa{
\|(a_n-i)^{-1}-(0-i)^{-1}\|&=\left \|i\int_0^{\infty}e^{-t}(e^{-ita_n}-1)\,{\rm{d}}t\right \|\\
&\le \int_0^{\infty}e^{-t}\|e^{-ita_n}-1\|\,{\rm{d}}t\stackrel{n\to \infty}{\to}0. 
}

Then by Lemma \ref{lem: NRT implies norm}, $\displaystyle \lim_{n\to \infty}\|a_n\|=0$ holds. 
Thus $\tau=\tau'$, and $A_{\rm{sa}}$ is {\rm{SUR}} by Theorem \ref{thm: SUR positive type}.   
\end{proof}

\begin{proof}[Proof of Theorem \ref{thm: UR iff finite-dim}]
Since $A$ is isomorphic as an additive group to $A_{\rm{sa}}\times A_{\rm{sa}}$, (i)$\Leftrightarrow $(ii) holds.\\
(ii)$\Rightarrow $(iv): Assume that $\dim A=\infty$. Let $B\subset A$ be a maximal abelian $*$-subalgebra of $A$. Then by Proposition \ref{prop: masa is infinitedim}, $\dim B=\infty$ holds. Then $B_{\rm{sa}}$ as a Banach space does not have cotype 2 by Lemma \ref{lem: cotypeC(X)}, so that $B_{\rm{sa}}$ is not {\rm{SUR}} by Lemma \ref{lem: URcotype2}. Therefore $A_{\rm{sa}}(\supset B_{\rm{sa}})$ is not {\rm{SUR}} either.\\
(iii)$\Rightarrow $(iv) We show the contrapositive. Assume that $\dim A=\infty$. Let $B$ be a maximal abelian ${\rm{C}}^*$-subalgebra of $A$. Then $\dim B=\infty$ by Proposition \ref{prop: masa is infinitedim}. By Lemma \ref{lem: infinite sep subalg}, $B$ contains a separable infinite-dimensional unital ${\rm{C}}^*$-subalgebra $C$. If $\mathcal{U}(C)_u$ were {\rm{SUR}}, then so would be $C_{\rm{sa}}$ by Proposition \ref{prop: key prop U(A) UR A UR}. By Lemma \ref{lem: UR then embeds to L0}, it follows that $C_{\rm{sa}}$ would be isomorphic as a topological vector space to a closed subspace of $L^0(\Omega,\mu)$, which implies that $C_{\rm{sa}}$ has cotype 2 by Lemma \ref{lem: URcotype2}. However, this contradicts Lemma \ref{lem: cotypeC(X)}. Therefore $\mathcal{U}(C)_u$, hence $\mathcal{U}(A)_u$ is not {\rm{SUR}}.\\
(iv)$\Rightarrow $(iii): By $\dim A<\infty$, $\mathcal{U}(A)_u=\mathcal{U}(A)_s$ is compact hence {\rm{SUR}}.\\
(iv)$\Rightarrow$(ii) Since $A_{\rm{sa}}$ is a finite-dimensional Banach space, it is isomorphic as an additive Polish group to some $\mathbb{R}^n\,(n\in \mathbb{R})$, which is SUR. 
\end{proof}
In \cite[The paragraph after Lemma 5.1]{Pestov18} (see also $\S$3.4 of the cited paper), Pestov implicitly mentioned that the Fredholm unitary group $\mathcal{U}_{\infty}(\ell^2)$ (equipped with the norm topology) is SUR (recall by Theorem \ref{thm: SUR positive type} that SUR is equivalent to having a family of continuous positive definite functions separating points and closed sets). We remark that this is not the case. 
We show that it is not the case. More generally, we show the following (take $A=\mathbb{B}(\ell^2),\ I=\mathbb{K}(\ell^2)$). 
\begin{corollary}
Let $A$ be a unital ${\rm{C}}^*$-algebra, $I$ be a separable two-sided closed ideal of $A$. 
Then the group $\mathcal{U}_I(A)=\{u\in \mathcal{U}(A)\mid u-1\in I\}$ equipped with the norm topology is {\rm{SUR}} if and only if $I$ is finite-dimensional. 
\end{corollary}
\begin{proof} 
If $1\in I$, then $\mathcal{U}_I(A)=\mathcal{U}(I)$, which is SUR if and only if $\dim I<\infty$ by Theorem \ref{thm: UR iff finite-dim}. 
Thus we assume that $I$ is proper and $\mathcal{U}_I(A)$ is SUR. We will show that $\dim I<\infty$. 
Let $\widetilde{I}$ be the unitization of $I$, which 
can be identified with $I+\mathbb{C}1_A\subset A$. First, we observe that the map $\pi\colon \mathcal{U}_I(A)\ni u\mapsto [u]\in  \mathcal{U}(\widetilde{I})/\mathbb{T}$ is an isomorphism of topological groups, where $[u]$ is the class of $u\in \mathcal{U}_I(A)$ in the quotient group. Clearly $\pi$ is continuous. Let $u=u_0+z1\in \mathcal{U}(\widetilde{I})$, where $u_0\in I$ and $z\in \mathbb{C}$. Since $u^*u=(u_0^*+\overline{z}1)(u_0+z1)=(u_0^*u_0+zu_0^*+\overline{z}u_0)+|z|^21=1$, $|z|=1$ holds. Therefore $\overline{z}u\in \mathcal{U}(A)$ and $\overline{z}u-1=\overline{z}u_0\in I$. This shows that $\overline{z}u\in \mathcal{U}_I(A)$ and $\pi(\overline{z}u)=[u]$, whence $\pi$ is surjective. For $u\in \mathcal{U}_I(A)$, if $[u]=1$, then there exists $w\in \mathbb{T}$ such that $wu=1$, while $u-1=\overline{w}-1\in I$. Since $1\notin I$, this shows that $w=1$ and $u=1$ holds. Therefore $\pi$ is injective. To see that $\pi$ is open, it suffices (since all the topological groups involved are Polish) to show that whenever a sequence $(u_n)_{n=1}^{\infty}$ in $\mathcal{U}_I(A)$ satisfies $[u_n]\stackrel{n\to \infty}{\to}[1]$, we have $u_n\stackrel{n\to \infty}{\to}1$. Fix such $(u_n)_{n=1}^{\infty}$ and write $u_n=v_n+1\ (v_n\in I,\,n\in \mathbb{N})$. Then for each $n\in \mathbb{N}$, there exists $z_n\in \mathbb{T}$ such that $\lim_n\|z_nu_n-1\|=\lim_n\|z_nv_n-(1-z_n)\|=0$. If $z_n-1$ does not tend to 0 as $n\to \infty$, then there exists a subsequence $(z_{n_k})_{k=1}^{\infty}$ and $z_0\neq 1$ such that $z_{n_k}-z_0\stackrel{k\to \infty}{\to}0$. For each $k,\ell\in \mathbb{N}$, 
\eqa{
\|z_{n_k}v_{n_k}-z_{n_{\ell}}v_{n_{\ell}}\|&\le \|z_{n_k}v_{n_k}-(1-z_{n_k})\|+\|(1-z_{n_k})-(1-z_{n_{\ell}})\|+\|(1-z_{n_{\ell}})-z_{n_{\ell}}v_{n_{\ell}}\|\\
&\stackrel{k,\ell\to \infty}{\to}0.
} 
Therefore the limit $v=\lim_{k}z_{n_k}v_{n_k}\in I$ exists. This implies that 
\[0=\lim_{k\to \infty}\|z_{n_k}v_{n_k}-(1-z_{n_k})\|=\|v-(1-z_0)\|.\]
Then because $v\in I$ and $1-z_0\in \mathbb{C}1$, we have $z_0=1$, a contradiction. Therefore $\pi$ is open and consequently, $\mathcal{U}(\widetilde{I})/\mathbb{T}$ is SUR. Since the map $\mathcal{U}(\widetilde{I})\ni u=u_0+z1\mapsto ([u],z)\in \mathcal{U}(\widetilde{I})/\mathbb{T}\times \mathbb{T}$ is an embedding of the Polish group $\mathcal{U}(\widetilde{I})$ onto a closed subgroup of $\mathcal{U}(\widetilde{I})/\mathbb{T}\times \mathbb{T}$, which is SUR, it follows that $\mathcal{U}(\widetilde{I})$ is SUR, whence $\widetilde{I}$ is finite-dimensional by Theorem \ref{thm: UR iff finite-dim}. 
\end{proof}

\subsection{Norm Unitary Representability (NUR)}\label{subsec: NUR}
Next, we turn our attention to NUR groups. First, the following result follows from a well-known fact that Banach Lie groups do 
not have small subgroups (see e.g., \cite[Theorem III.4.1]{Omori97}). We include a simple proof for the reader's convenience.   
\begin{proposition}\label{prop: small subgroups} Let $H$ be a Hilbert space and $G$ be a topological group. If $G$ has small subgroups, then there is no norm-continuous injective homomorphism from $G$ into $\mathcal{U}(H)$. 
\end{proposition}
\begin{proof}
Assume by contradiction that $\pi\colon G\to \mathcal{U}(H)$ is a norm-continuous injective homomorphism. 
Since $G$ has small subgroups, for each $U\in \mathscr{N}(G)$, there exists a subgroup $G_U\neq \{1_G\}$ contained in $U$. 
In particular, there exists an element $g_U\in G_U\setminus \{1_G\}$. Since $\pi$ is injective, its spectrum $\sigma(\pi(g_U))$ contains an element $\lambda_U\in \mathbb{T}\setminus \{1\}$. Then there exists $k_U\in \mathbb{N}$ such that $\text{Re}(\lambda_U^{k_U})\le 0$ (we may assume that $\lambda_U=e^{i\theta_U}$ with $0<|\theta_U|<\frac{\pi}{2}$. Then we can choose $k_U\in \mathbb{N}$ such that $\frac{\pi}{2|\theta_U|}\le k_U\le \frac{\pi}{|\theta_U|}$). Since $G_U$ is a subgroup of $G$, we have $g_U^{k_U}\in G_U\subset U$. This shows that $\displaystyle \lim_{U\in \mathscr{N}(G)}g_U^{k_U}=1_G$, whence $\displaystyle \lim_{U\in \mathscr{N}(G)}\|\pi(g_U^{k_U})-1\|=0$. However, for each $U\in \mathscr{N}(G)$, we have 
\[\|\pi(g_U^{k_U})-1\| \ge |\lambda_U^{k_U}-1|\ge 1,\]
which is a contradiction. 
\end{proof}
The above proposition shows in particular that the Polish groups of the form $\prod_{n=1}^{\infty}G_n$, where each $G_n$ is a non-trivial SUR Polish group, are not NUR, although they are SUR (since SUR passes to countable products). 
Another immediate consequence is that for an infinite-dimensional von Neumann algebra $M$, the group $\mathcal{U}(M)_s$ is never NUR. 

\begin{corollary}\label{cor: U(M)_s not NUR} For a von Neumann algebra $M$ with separable predual, the Polish group $\mathcal{U}(M)_s$ is {\rm{NUR}}, if and only if $M$ is finite-dimensional. 
\end{corollary} 
\begin{proof}
If $\dim M<\infty$, then $\mathcal{U}(M)_s=\mathcal{U}(M)_u$ is compact, so it is NUR. Conversely, if $M$ is infinite-dimensional, let $A$ be a maximal abelian subalgebra of $M$ which is an infinite-dimensional abelian von Neumann algebra with separable predual (cf. Proposition \ref{prop: masa is infinitedim}). Then $A$ has a direct summand $A_0$ isomorphic to either $\ell^{\infty}$ or $L^{\infty}([0,1])$. Let $J_n=(2^{-n},2^{-n+1}]\ (n\in \mathbb{N})$. Then the set of all $f\in L^{\infty}([0,1])$ which are essentially constant on each $J_n\ (n\in \mathbb{N})$ forms a von Neumann subalgebra isomorphic to $\ell^{\infty}$ and $\mathcal{U}(\ell^{\infty})_s=\mathbb{T}^{\mathbb{N}}$ has small subgroups. Thus, in both cases the group $\mathcal{U}(A_0)_s$, hence $\mathcal{U}(M)_s$, has small subgroups. The conclusion now follows from Proposition \ref{prop: small subgroups}.    
\end{proof}
\begin{proposition}\label{prop: C[0,1] is NUR} Let $C([0,1],\mathbb{R})$ be the Banach space of all continuous real valued functions on $[0,1]$. Then $C([0,1],\mathbb{R})$ as an additive Polish group is {\rm{NUR}}. 
\end{proposition}
\begin{proof} Let $G=C([0,1],\mathbb{R})$. 
Let $m\colon G\to \mathbb{B}(L^2([0,1]))$ be the map given by $[m(f)g](x)=f(x)g(x),\ f\in C([0,1],\mathbb{R}),\,g\in L^2([0,1])$ and $x\in [0,1]$. Fix an enumeration $(t_n)_{n=1}^{\infty}$ of $\mathbb{Q}\cap [-1,1]$. Let $H=\bigoplus_{n=1}^{\infty}L^2([0,1])$ and $\pi\colon G\to \mathcal{U}(H)$ be the unitary representation given by $\pi(f)=(e^{2\pi it_nm(f)})_{n=1}^{\infty},\ f\in G$. Then $\pi$ is a topological group isomorphism of $G$ onto $\pi(G)$ endowed with the norm topology. To see this, let $(f_n)_{n=1}^{\infty}$ be a sequence in $G$. If $\|f_n\|_{\infty}\stackrel{n\to \infty}{\to}0$, then 
\[\|\pi(f_k)-1\|=\sup_{n\in \mathbb{N}}\|e^{2\pi it_nm(f_k)}-1\|=\sup_{|t|\le 1}\max_{x\in [0,1]}|e^{2\pi itf_k(x)}-1|\stackrel{k\to \infty}{\to}0. \]
Also if $\|\pi(f_k)-1\|\stackrel{k\to \infty}{\to}0$, then for each $t\in [-1,1]$, one has $\|e^{2\pi itm(f_k)}-1\|\stackrel{k\to \infty}{\to}0$, and by the group property, the same convergence holds for all $t\in \mathbb{R}$. Moreover, because the map $t\mapsto e^{-t}e^{-itm(f_k)}$ is norm-continuous, the dominated convergence theorem implies that 
\eqa{
(m(f_k)-i)^{-1}-(0-i)^{-1}&=i\int_0^{\infty}e^{-t}(e^{-itm(f_k)}-1)\,{\rm{d}}t\stackrel{k\to \infty}{\to}0\ \ \ (\text{norm}).
} 
Then by Lemma \ref{lem: NRT implies norm}, we have $\displaystyle \lim_{k\to \infty}\|m(f_k)\|=\lim_{k\to \infty}\|f_k\|_{\infty}=0$. 
Therefore $\pi$ is a topological group isomorphism of the Polish group $G$ onto $\pi(G)$ endowed with the norm topology. In particular, $\pi(G)$ is norm-separable whence $\overline{\pi(G)}$ is a Polish group in the norm topology, of which $\pi(G)$ is a Polish subgroup. Therefore $\pi(G)$ is norm-closed.  
\end{proof}
\begin{corollary}\label{cor: sep Banach is NUR}
Every separable Banach space is {\rm{NUR}}.  
\end{corollary}
\begin{proof}
Since $C([0,1])$ is isomorphic as a topological group to $C([0,1],\mathbb{R})\times C([0,1],\mathbb{R})$, $C([0,1])$ is NUR. Since every separable Banach space admits a linear isometric embedding into $C([0,1])$ by Banach--Mazur Theorem (see e.g., \cite[Theorem 1.4.3]{AlbiacKalton06}), the claim follows from Proposition \ref{prop: C[0,1] is NUR}.  
\end{proof}
\begin{remark}
Corollary \ref{cor: sep Banach is NUR} implies that there exist non-locally compact Polish groups which are both SUR and NUR. Namely, $\ell^p\ (1\le p\le 2)$ as an additive group is such an example (the fact that they are SUR is a classical result of Shoenberg \cite{Schoenberg}). We do not know if there exists a Polsih group which is both NUR and SUR but fails to be of finite type (cf. \cite{AMTT}). 
\end{remark}
Next, we show a more interesting example: $\mathcal{U}_1(\ell^2)$ is NUR. To state the theorem precisely, we need the theory of Fermion Fock space. We briefly recall some basic notations and facts about it and refer the reader to \cite[Chapter 6]{Ar18} for details.

Let $H$ be a Hilbert space.
For each natural number $p\geq 1$, let $S_p$ be the symmetric group of degree $p\geq 1$, which acts on the $p$-fold tensor product Hilbert space $\otimes^pH$ by
\[
U(\sigma)\xi_1\otimes\cdots\otimes\xi_p = \xi_{\sigma(1)}\otimes\cdots\otimes\xi_{\sigma(p)}, \ \ \ \ \ \sigma\in S_p,\ \ \xi_1,\cdots,\xi_p\in\ H.
\]
Set
\[
A_p = \frac{1}{p!}\sum_{\sigma\in S_p}{\rm sgn}{(\sigma)}U(\sigma).
\]
Then $A_p$ is an orthogonal projection, and its range $\wedge^pH=A_p\otimes^pH$ is called the $p$-fold anti-symmetric tensor product of $H$.
Put $\wedge^0H=\mathbb{C}$.
The Fermion Fock space over $H$ is defined by
\[
\mathscr{F} = \mathscr{F}_{\rm f}(H) = \bigoplus_{p=0}^{\infty}\wedge^p(H).
\]
The vector $\Omega=(1,0,0,\cdots)\in\mathscr{F}$ is called the Fock vacuum.

We next define the annihilation and creation operators.
For each  vector $\xi\in H$, we define the creation operator $a^{\dagger}(\xi)$ as a bounded linear operator on $\mathscr{F}$ so that 
\[
a^{\dagger}(\xi)\psi = \sqrt{p+1}A_{p+1}(\xi\otimes\psi),\ \ \ \ \ p\geq0,\ \ \psi\in\otimes^pH.
\]
The annihilation operator $a(\xi)$ is defined by $a(\xi)=a^{\dagger}(\xi)^*$, 
and hence $a(\xi)^*=a^{\dagger}(\xi)$.
They satisfy the canonical anti-commutation relations (CARs):
\[
\{a(\xi),a(\eta)^*\}=\langle\xi,\eta\rangle,\ \ \ \ \ 
\{a(\xi),a(\eta)\} = \{a(\xi)^*,a(\eta)^*\} =0
\]
for all $\xi,\eta\in H$.
Here, for two bounded linear operators $A,B$, the anti-commutator $\{A,B\}=AB+BA$ is used.
It is known that $a(\xi)^2=(a(\xi)^*)^2=0$ and $\|a(\xi)\|_{\infty}=\|\xi\|_{H}$.
Moreover, if $H$ is separable infinite-dimensional and $\{\xi_n\}_{n=1}^{\infty}$ is an orthogonal basis for $H$, then
\[
\left\{a(\xi_{n_1})^*\cdots a(\xi_{n_{\ell}})^*\Omega \mid \ell\geq 0,\ n_1<n_2<\cdots <n_{\ell}\right\}
\]
is an orthogonal basis for $\mathscr{F}$.

Let ${\rm CAR}(H)$ be the $C^*$-algebra generated by $\{a(\xi)\mid\xi\in H\}$,
which is called the CAR algebra.
It is known that if $H$ is separable infinite-dimensional, then ${\rm CAR}(H)$ is $*$-isomorphic to the UHF algebra $M_{2^{\infty}}$ of type $2^{\infty}$.

For any $u\in\mathcal{U}(H)$ and $p\geq 1$, we define a unitary operator $\otimes^pu$ on $\otimes^pH$ by
\[
(\otimes^pu)\xi_1\otimes\cdots\otimes\xi_p = u\xi_1\otimes\cdots\otimes u\xi_p,\ \ \ \ \ 
\xi_1,\cdots,\xi_p\in H.
\]
Then $\wedge^pH$ reduces $\otimes^pu$.
We denote by $\wedge^pu$ its reduced part.
Set $\wedge^0u=1$.
The second quantization $\Gamma(u)$ of $u$ is defined by
\[
\Gamma(u) = \bigoplus_{p=0}^{\infty}\wedge^pu.
\]
Note that $\Gamma(u)$ is a unitary on $\mathscr{F}$, and thus the map $\Gamma\colon \mathcal{U}(H)_s\to \mathcal{U}(\mathscr{F})_s$ 
is a continuous unitary representation.
It also follows that
\[
\Gamma(u)a(\xi)\Gamma(u)^* = a(u\xi),\ \ \ \ \ u\in\mathcal{U}(H),\ \ \xi\in H.
\]
We can now state the theorem. 
\begin{theorem}\label{U1 is NUR}
$\mathcal{U}_1(\ell^2)$ is isomorphic to a closed subgroup of 
$\mathcal{U}({\rm CAR}(\ell^2))_u$ via the second quantization map $\Gamma$.
In particular, $\mathcal{U}_1(\ell^2)$ is {\rm{NUR}}.
\end{theorem}
We need the following three lemmas.

\begin{lemma}\label{first lemma}
If a sequence $\{x_n\}_{n=1}^{\infty}$ of bounded operators on $H$ satisfies 
$\sum_{n=1}^{\infty}\|x_n\|_{\infty}<\infty$, then the limit
\[
X=\lim_{N\to\infty}\prod_{n=1}^N(1+x_n) = \lim_{N\to\infty}(1+x_1)(1+x_2)\cdots(1+x_N)
\]
exists in the operator norm topology and that the inequality
\[
\|X-1\|_{\infty}\leq \exp{\left(\sum_{n=1}^{\infty}\|x_n\|_{\infty}\right)}-1
\]
holds.
\end{lemma}

\begin{proof}
For natural numbers $M>N$, we have
\begin{align*}
\left\|\prod_{n=1}^M(1+x_n)-\prod_{n=1}^N(1+x_n)\right\|_{\infty} 
&\leq \left\|\prod_{n=1}^N(1+x_n)\right\|_{\infty}\left\|\prod_{n=N+1}^M(1+x_n)-1\right\|_{\infty}\\
&\leq\left\{\prod_{n=1}^N(1+\|x_n\|_{\infty})\right\} \left\{\prod_{n=N+1}^M(1+\|x_n\|_{\infty})-1\right\}\\
&\to 0\ \ \ \ \ (M,N\to\infty).
\end{align*}
Hence the limit
$X=\lim_{N\to\infty}\prod_{n=1}^N(1+x_n)$
exists in the operator norm topology.
Similarly, we obtain
\[
\|X-1\|_{\infty}
\leq \prod_{n=1}^{\infty}(1+\|x_n\|_{\infty})-1
\leq \prod_{n=1}^{\infty}\exp{(\|x_n\|_{\infty})}-1
\leq \exp{\left(\sum_{n=1}^{\infty}\|x_n\|_{\infty}\right)}-1
\]
This completes the proof.
\end{proof}

\begin{lemma}\label{second lemma}
For any $u,v\in\mathcal{U}(\ell^2)$, we have 
$\|u-v\|_{\infty}\leq\|\Gamma(u)-\Gamma(v)\|_{\infty}$.
\end{lemma}

\begin{proof}
A direct computation shows that
\[
\|u-v\|_{\infty} = \sup_{\|\xi\|=1}\|u\xi-v\xi\|_{\ell^2} 
= \sup_{\|\xi\|=1}\left\|\left\{\Gamma(u)-\Gamma(v)\right\}a(\xi)^*\Omega\right\|_{\mathscr{F}}
\leq \|\Gamma(u)-\Gamma(v)\|_{\infty},
\]
which is the desired result.
\end{proof}

Before going to the last lemma, recall that $\mathcal{U}_{\infty}(\ell^2)$ is a closed subgroup of $\mathcal{U}(\ell^2)$ with respect to the norm topology. 

\begin{lemma}\label{third lemma}
Let $u\in\mathcal{U}_{\infty}(\ell^2)$, and let $0<\varepsilon<\pi/2$.
If $\|\Gamma(u)-1\|_{\infty}\leq (2-2\cos{\varepsilon})^{1/2}$, then
$u\in\mathcal{U}_1(\ell^2)$ and $\|u-1\|_1\leq2\varepsilon$ hold.
\end{lemma}

\begin{proof}
Let 
\[
u=\sum_{n=1}^{\infty}\lambda_n\langle\xi_n,\cdot\rangle\xi_n,\ \ \ \ \ 
|\lambda_n|=1,\ \ \ \ \ \{\xi_n\}_{n=1}^{\infty}\textrm{\ is an orthonormal basis for $\ell^2$}
\]
be the spectral resolution of $u\in\mathcal{U}_{\infty}(\ell^2)$.
By Lemma \ref{second lemma}, we have $\|u-1\|_{\infty}\leq (2-2\cos{\varepsilon})^{1/2}$,
whence we can write $\lambda_n$ as $\lambda_n=e^{i\theta_n}$ with some real number
$|\theta_n|\leq\varepsilon$.
Put $D_+=\{n\in\mathbb{N}\mid\theta_n\geq0\}$ and $D_-=\{n\in\mathbb{N}\mid\theta_n<0\}$.

We claim that $\sum_{n\in D_+}\theta_n\leq \varepsilon$.
To show the claim, we may assume that $\sharp D_+=\infty$, and let $\{\theta_{n_{\ell}}\}_{\ell=1}^{\infty}$ be the subsequence of $\{\theta_n\}_{n=1}^{\infty}$ consisting of those $\theta_n$ for which $n\in D_+$ holds.
We show, by induction on $L\in\mathbb{N}$, that for every $L\in\mathbb{N}$,
the inequality $\sum_{\ell=1}^{L}\theta_{n_{\ell}}\leq\varepsilon$ holds.
Once we prove this, the claim follows.
The case $L=1$ follows from the definition of $\theta_n$.
We next assume that the inequality holds for some $L\in\mathbb{N}$.
Since
\begin{align*}
(2-2\cos{\varepsilon})^{\frac{1}{2}} &\geq \|\Gamma(u)-1\|_{\infty} 
\geq \left\|(\Gamma(u)-1)a(\xi_{n_1})^*\cdots a(\xi_{n_{L+1}})^*\Omega\right\|_{\mathscr{F}}\\
&= \left\|a(u\xi_{n_1})^*\cdots a(u\xi_{n_{L+1}})^*\Omega-a(\xi_{n_1})^*\cdots a(\xi_{n_{L+1}})^*\Omega\right\|_{\mathscr{F}}\\
&= \left\|a(\exp{(i\theta_{n_1})}\xi_{n_1})^*\cdots a(\exp{(i\theta_{n_{L+1}})}\xi_{n_{L+1}})^*\Omega-a(\xi_{n_1})^*\cdots a(\xi_{n_{L+1}})^*\Omega\right\|_{\mathscr{F}}\\
&= \left|\exp{\left(i\sum_{\ell=1}^{L+1}\theta_{n_{\ell}}\right)}-1\right|,
\end{align*}
we get
\[
2\pi m-\varepsilon \leq \sum_{\ell=1}^{L+1}\theta_{n_{\ell}} \leq 2\pi m +\varepsilon
\]
for some integer $m$.
But we know that
\[
0\leq \sum_{\ell=1}^{L}\theta_{n_{\ell}} \leq\varepsilon,\ \ \ \ \ 
0\leq\theta_{n_{L+1}}\leq\varepsilon,\ \ \ \ \ 0<\varepsilon<\frac{\pi}{2}.
\]
Hence $\sum_{\ell=1}^{L+1}\theta_{n_{\ell}} \leq\varepsilon$ holds, which is the case $L+1$.
This finishes the proof of the claim.

Similarly, we can show that $\sum_{n\in D_-}\theta_n\leq \varepsilon$.
This together with the claim implies that $\sum_{n=1}^{\infty}|\theta_n|\leq2\varepsilon$.
Recall the elementary inequality $0\leq 2-2\cos{x}\leq x^2\,(x\in \mathbb{R})$.
Then we obtain
\[
\|u-1\|_1=\sum_{n=1}^{\infty}|e^{i\theta_n}-1| = \sum_{n=1}^{\infty}(2-2\cos{\theta_n})^{\frac{1}{2}}
\leq\sum_{n=1}^{\infty}|\theta_n| \leq 2\varepsilon,
\]
so that $u\in\mathcal{U}_1(\ell^2)$ and $\|u-1\|_1\leq2\varepsilon$ hold.
\end{proof}

\begin{proof}[Proof of Theorem \ref{U1 is NUR}]
We split the proof into two steps.

{\bf Step 1.} For any $u\in\mathcal{U}_1(\ell^2)$, we have $\Gamma(u)\in{\rm CAR}(\ell^2)$ and $\|\Gamma(u)-1\|_{\infty}\leq e^{\|u-1\|_1}-1$ holds.
In particular, the map
\[
\Bigl(\mathcal{U}_1(\ell^2),\|\cdot\|_1\Bigr) \to \Bigl(\mathcal{U}({\rm CAR}(\ell^2)), \|\cdot\|_{\infty}\Bigr),\ \ \ \ \ u\mapsto\Gamma(u)
\]
is continuous.

To see this, let
\[
u=\sum_{n=1}^{\infty}\lambda_n\langle\xi_n,\cdot\rangle\xi_n,\ \ \ \ \ 
|\lambda_n|=1,\ \ \ \ \ \{\xi_n\}_{n=1}^{\infty}\textrm{\ is an orthonormal basis for $\ell^2$}
\]
be the spectral resolution of $u\in\mathcal{U}_1(\ell^2)$.
Since
\[
\sum_{n=1}^{\infty}\|(\lambda_n-1)a(\xi_n)^*a(\xi_n)\|_{\infty}
\leq \sum_{n=1}^{\infty}|\lambda_n-1| = \|u-1\|_1 < \infty,
\]
we can apply Lemma \ref{first lemma} to the sequence $((\lambda_n-1)a(\xi_n)^*a(\xi_n))_{n=1}^{\infty}$ of bounded operators on $\mathscr{F}$ to get that the limit
\[
X_u = \lim_{N\to\infty}\prod_{n=1}^N\{1+(\lambda_n-1)a(\xi_n)^*a(\xi_n)\}\in{\rm CAR}(\ell^2)
\]
exists in the operator norm topology, and the inequality
$\|X_u-1\|_{\infty}\leq e^{\|u-1\|_1}-1$ holds.
Thus it is enough to show that $\Gamma(u)=X_u$.
For this, let us consider
\[
\Gamma(u)a(\xi_{\ell_1})^*\cdots a(\xi_{\ell_j})^*\Omega 
=a(u\xi_{\ell_1})^*\cdots a(u\xi_{\ell_j})^*\Omega
\]
for any natural numbers $\ell_1<\ell_2<\cdots<\ell_j$ with $j\geq 1$.
Note that, by the CARs, if $n\not=\ell$, we have
\[
a(\xi_n)^*a(\xi_n)\cdot a(\xi_{\ell})^* = -a(\xi_n)^*a(\xi_{\ell})^*a(\xi_n) 
= a(\xi_{\ell})^*\cdot a(\xi_n)^*a(\xi_n),
\]
and if $n=\ell$, we have
\[
a(\xi_n)^*a(\xi_n)\cdot a(\xi_{\ell})^* = -a(\xi_n)^*a(\xi_{\ell})^*a(\xi_n)+a(\xi_n)^* = a(\xi_n)^*.
\]
Hence for all $N>\ell_j$, we obtain
\begin{align*}
\prod_{n=1}^N\{1+(\lambda_n-1)a(\xi_n)^*a(\xi_n)\}\cdot &a(\xi_{\ell_1})^*\cdots a(\xi_{\ell_j})^*\Omega 
= \lambda_{\ell_1}a(\xi_{\ell_1})^*\cdot\lambda_{\ell_2}a(\xi_{\ell_2})^*\cdots \lambda_{\ell_j}a(\xi_{\ell_j})^*\Omega\\
&= a(u\xi_{\ell_1})^*\cdots a(u\xi_{\ell_j})^*\Omega
= \Gamma(u)\cdot a(\xi_{\ell_1})^*\cdots a(\xi_{\ell_j})^*\Omega.
\end{align*}
Therefore $\Gamma(u)=X_u$ follows.

{\bf Step 2.} the image $\Gamma\Bigl(\mathcal{U}_1(\ell^2)\Bigr)$ is closed and the map $\Gamma$ is a homeomorphism.

To see this, take any sequence $(u_n)_{n=1}^{\infty}$ of $\mathcal{U}_1(\ell^2)$ and any bounded operator $T$ on $\mathscr{F}$ satisfying $\|\Gamma(u_n)-T\|\stackrel{n\to \infty}{\to}0$.
It is enough to show that there exists $u\in\mathcal{U}_1(\ell^2)$ such that 
$T=\Gamma(u)$ and $\|u_n-u\|_1\stackrel{n\to \infty}{\to}0$. 
By Lemma \ref{second lemma}, we have
\[
\|u_m-u_n\|_{\infty} \leq \|\Gamma(u_m)-\Gamma(u_n)\|_{\infty}\stackrel{n,m\to \infty}{\longrightarrow}0,
\]
thus there exists $u\in\mathcal{U}_{\infty}(\ell^2)$ such that 
$\|u_n-u\|_{\infty}\stackrel{n\to \infty}{\to}0$.
Since the map $\Gamma\colon \mathcal{U}(\ell^2)_s\to \mathcal{U}(\mathscr{F})_s$
is continuous, $\Gamma(u_n)\stackrel{n\to \infty}{\to}\Gamma(u)$ (SOT).
Thus $T=\Gamma(u)$ and
\[
\|\Gamma(u^*u_n)-1\|_{\infty} = \|\Gamma(u_n)-\Gamma(u)\|_{\infty} \stackrel{n\to \infty}{\to}0.
\] 
This, together with Lemma \ref{third lemma} implies that there exists $n_0\in\mathbb{N}$ such that for any $n\geq n_0$, we have $u^*u_n\in\mathcal{U}_1(\ell^2)$ and $\|u^*u_n-1\|_1\stackrel{n\to \infty}{\to}0$.
Since $u_n\in\mathcal{U}_1(\ell^2)$, we conclude that 
$u\in\mathcal{U}_1(\ell^2)$ and $\|u_n-u\|_1\stackrel{n\to \infty}{\to}0$.
This finishes the proof.
\end{proof}
\begin{remark}
We do not know if $\mathcal{U}_1(\ell^2)$ is SUR. We also remark that Boyer \cite{Boyer80,Boyer88} studied a family of strongly continuous unitary representations of $\mathcal{U}_1(\ell^2)$ and $\mathcal{U}_2(\ell^2)$ . 
\end{remark}
\section{Boundedness, Property (OB), (T) and (FH)}\label{sec: boundedness}
In this section, we consider structural questions about specific groups of unitaries. 
As is explained in $\S$\ref{sec: preliminaries}, (T)$\Rightarrow$(FH) and bounded$\Rightarrow $(OB)$\Rightarrow$(FH) hold in general. 
\subsection{Boundedness and Property (OB)}
We consider the following question: let $A$ be a unital C$^*$-algebra. Atkin \cite{Atkin91} has shown that $\mathcal{U}(\ell^2)_u$ (hence also $\mathcal{U}(\ell^2)_s$) is bounded. 
When is $\mathcal{U}(A)_u$ bounded?
By a simple argument using spectral theory one can show that in fact $\mathcal{U}(M)_u$ is bounded for every von Neumann algebra and for every unital AF-algebra. This is a folklore result (see e.g., \cite[Proposition 1.3]{Neeb14}), but because the proof can be used to prove the boundedness for more general unitary groups, we provide a proof.   
\begin{proposition}[Folklore]\label{prop: boundedness of U(A) for AF}Let $A$ be a unital {\rm{AF}} algebra. Then the group $\mathcal{U}(A)_u$ is bounded. 
\end{proposition}
\begin{remark}
 Dowerk \cite{Dowerk18} has shown that the unitary group of a II$_1$ factor or a properly infinite von Neumann algebra has a much stronger property called the Bergman's property. Namely, whenever $W_1\subset W_2\subset \cdots$ is an exhaustive increasing sequence of proper symmetric subsets of $\mathcal{U}(M)$ (no topological condition is imposed on $W_n$'s), there exist $n,k\in \mathbb{N}$ such that $W_n^k=\mathcal{U}(M)$ holds. The property (OB) is considered to be (and initially introduced as such by Rosendal \cite{Rosendal09}) a topological version of the Bergman's property. 
\end{remark}
We need a preparation. 
\begin{definition}\label{def: theta(r)}
For $0<r<2$, let $\theta(r)$ be the unique $\theta\in (0,\pi)$ satisfying $\cos \theta =1-\frac{r^2}{2}$. 
We also define $I_r=[0,\theta(r))\cup (2\pi-\theta(r),2\pi), J_r=[\theta(r),2\pi-\theta(r)]=[0,2\pi)\setminus I_r$ and $n_0(r)=\left [\frac{2(\pi-\theta(r))}{\theta(r)}\right ]+1\in \mathbb{N}$..  
\end{definition} 
The following Lemma is also well-known (see e.g., \cite[Proposition 1.3]{Neeb14}). 
\begin{lemma}\label{lem: bounded for vNa}
Let $M$ be a von Neumann algebra on a Hilbert space $H$. Then $\mathcal{U}(M)_u$ is bounded. More precisely, for any $0<r<2$, there exists $n=n(r)\in \mathbb{N}$ (which depends only on $r$, not on $M$) such that $\mathcal{U}(M)=V_r^{n(r)}$, where $V_r=\{u\in \mathcal{U}(M);\ \|u-1\|<r\}$. 
\end{lemma}
\begin{proof}
Write $G=\mathcal{U}(M)_u$. Fix $0<r<2$.  
Observe that for $\theta\in [0,2\pi)$, we have 
$|e^{i\theta}-1|^2=2(1-\cos \theta)$, so that $|e^{i\theta}-1|<r\Leftrightarrow \theta\in I_r$. 
Let $u\in G$, and let $u=\int_0^{2\pi}e^{i\theta}\,\text{d}E(\theta)$ be the spectral resolution of $u$. 
 Then we have $\frac{2\pi -\theta(r)}{n_0(r)}<\theta(r)$. Note that for $\theta\in J_r$, we have $\frac{\theta}{n_0(r)}\in I_r$ by the choice of $n_0(r)$. 
Define $u_1,u_2\in G$ by 
\[u_1=\int_{I_r}e^{i\theta}\,\text{d}E(\theta)+E(J_r),\ \ u_2=\int_{J_r}e^{i\frac{\theta}{n_0(r)}}\,\text{d}E(\theta)+E(I_r).\]
Then $u_1,u_2\in V_r$, and $u=u_1u_2^{n_0(r)}\in V_r^{n_0(r)+1}$. This shows that $\mathcal{U}(M)=V_r^{n_0(r)+1}$. 
\end{proof}

\begin{proof}[Proof of Proposition \ref{prop: boundedness of U(A) for AF}] 
Fix a faithful representation $A\subset \mathbb{B}(H)$ of $A$ on a Hilbert space $H$. 
Write $G=\mathcal{U}(A)_u$. Let $V\in \mathscr{N}(G)$. Then there exists $0<r<2$ such that $V_r=\{u\in G; \|u-1\|<r\}\subset V$ holds. 
Since $A$ is AF, it is well-known that there exists a finite-dimensional unital $*$-subalgebra $B\subset A$ and $v\in \mathcal{U}(B)$ such that $\|u-v\|<r$. 
We apply Lemma \ref{lem: bounded for vNa} to $\mathcal{U}(B)$ to find $n(r)\in \mathbb{N}$ (depending only on $r$, independent of $B$) such that $\mathcal{U}(B)=(V_r\cap B)^{n(r)}$. 
Also, $\|1-uv^*\|=\|u-v\|<r$ implies that $uv^*\in V_r$. Therefore $u=(uv^*)v\in V_r(V_r\cap B)^{n(r)}\subset V_r^{n(r)+1}$. This shows that $G=V_r^{n(r)+1}$ and $G$ is bounded.
\end{proof}
 On the other hand, not every such group is bounded:
\begin{example}\label{ex: C(T) is not bounded}
Let $A=C(\mathbb{T})$. Then $G=\mathcal{U}(A)_u$ is not bounded. 
\begin{proof}
Let $V=\{u\in G; \|u-1\|_{\infty}<2\}$. Then there exists $\theta_0\in (0,\pi)$ such that $\sigma(u)\cap \{e^{i\theta};\ \theta\in [\theta_0,2\pi-\theta_0]\}=\emptyset$ for all $u\in V$. Assume by contradiction that $G=FV^n$ for some $F=\{f_1,\dots,f_k\}\subset G$ and $n\in \mathbb{N}$. 
Recall that $\pi_1(\mathbb{T},1)=\mathbb{Z}$ is generated by $\iota(z)=z$. We denote by $[f]\in \pi_1(\mathbb{T})$ represented by a loop $f\colon \mathbb{T}\to \mathbb{T}$ with $f(1)=1$ (we identify $f\in G$ such that $f(1)=1$, with the loop $c_f\colon [0,1]\to \mathbb{T}=\{z\in \mathbb{C};\ |z|=1\}$ given by 
$c_f(t)=f(e^{2\pi it})\ (t\in [0,1])$). 

Recall also that since $\mathbb{T}$ is a topological group, the product $[f]+[g]$ in $\pi_1(\mathbb{T})$ of loops $f,g\colon \mathbb{T}\to \mathbb{T} (f(1)=g(1)=1)$ corresponds to the class $[f\cdot g]$ represented by the product loop $z\mapsto f(z)g(z)$. 
Now, let $g\in V$ with $g(1)=1$. Then regarding $g$ as a loop in $\mathbb{T}$, we show that $[g]=0$. Indeed, if there were $n\in \mathbb{Z}\setminus \{0\}$ for which $[g]=n[\iota]=[\iota^n]$, there would exist a continuous map $\tilde{g}\colon [0,1]\to \mathbb{R}$ such that $\exp (2\pi i\tilde{g}(t))=g(e^{2\pi it})\ (t\in [0,1))$ and $\tilde{g}(1)=n\neq 0$. By considering $g^*$ instead of $g$ if necessary, we may assume that $n\ge 1$.  Since $\tilde{g}$ is continuous, there exists by mean value theorem a $t_0\in \mathbb{R}$ such that $\tilde{g}(t_0)=\tfrac{1}{2}$. Thus $g(e^{2\pi it_0})=\exp(2\pi i\tilde{g}(t_0))=-1$. But this contradicts the assumption that $g\in V$. Therefore $[g]=0$.  
Let $f_i(1)=z_i\ (1\le i\le k)$ and let $n_i\in \mathbb{Z}$ be given by $[\tilde{f}_i]=n_i[\iota]$, where $\tilde{f}_i(z)=f_i(z)\overline{z_i}$. Fix $N\in \mathbb{Z}\setminus \{n_1,\dots,n_k\}$. We show that $f=\iota^N\notin FV^n$. Assume by contradiction that $f\in FV^n$. Choose $i\in \{1,\dots,k\}$ and $g_1,\dots,g_n\in V$ such that 
$f(z)=f_i(z)g_1(z)\cdots g_n(z)\ (z\in \mathbb{T})$. Let $z_0=f(1), w_i=g_i(1)$, $\tilde{f}(z)=f(z)\overline{z_0}$, and $\tilde{g}_j(z)=g_j(z)\overline{w_j}(z\in \mathbb{T})\ (1\le j\le n)$. Then $z_0=z_i\prod_{j=1}^nw_j$, so that $\tilde{f}=\tilde{f}_i\prod_{j=1}^n\tilde{g}_j$ is the product of loops based at 1. Therefore 
\[[\tilde{f}]=N[\iota]=[\tilde{f}_i]+\sum_{j=1}^n[\tilde{g}_j]=n_i[\iota],\]
which contradicts $N\neq n_i$. Therefore $G=\mathcal{U}(C(\mathbb{T}))_u$ is not bounded. 
\end{proof}
\end{example}
Next we characterize abelian unital C$^*$-algebras having bounded unitary groups. 
\begin{proposition}\label{prop: boundedness for abelian U(A)} Let $X$ be a compact Hausdorff space. Then the following conditions are equivalent. 
\begin{list}{}{}
\item[{\rm{(i)}}] $\mathcal{U}(C(X))_u$ is bounded 
\item[{\rm{(ii)}}] $X$ is totally disconnected.
\end{list} \end{proposition}
\begin{remark}
In view of Theorem \ref{thm: bounded and cel} below, Proposition \ref{prop: boundedness for abelian U(A)} follows also from Phillips' work \cite[Corollary 3.3]{Phillips95}, 
where he shows that ${\rm{cel}}(C(X)\otimes M_n(\mathbb{C}))<\infty$ if and only if $X$ is totally disconnected (here, cel stands for the exponential length (see Definition \ref{def: exp length} below). Note that if $X$ is totally disconnected, then $C(X)\otimes M_n(\mathbb{C})$ is a unital AF algebra and therefore its unitary group $\mathcal{U}(C(X)\otimes M_n(\mathbb{C}))$ is connected. 
\end{remark}

\begin{proof}[Proof of Proposition \ref{prop: boundedness for abelian U(A)}] Let $G=\mathcal{U}(C(X))$.\\
(ii)$\Rightarrow$(i) Assume that $X$ is totally disconnected. Then $C(X)$ is a unital {\rm{AF}}-algebra. Thus $G$ is bounded by Proposition \ref{prop: boundedness of U(A) for AF}.\\
(i)$\Rightarrow$(ii) We show the contrapositive. Assume that $X$ is not totally disconnected. 
Then there exists a connected component $X_0\subset X$ with $|X_0|\ge 2$.  Let $V=\{u\in G;\ \|u-1\|<\sqrt{2}\}\in \mathscr{N}(G)$. Let $F=\{g_1,\dots,g_k\}\subset G$ be a finite set and $n\in \mathbb{N}$. We show that $G\neq FV^n$. Define $D=\{z\in \mathbb{T}; \text{arg}(z)\in (-\frac{\pi}{2},\frac{\pi}{2})\}$. 
Then $-i\log \colon D\ni e^{it}\mapsto t\in (-\frac{\pi}{2},\frac{\pi}{2})$ is a continuous function. For each $g\in V$, we have that $g|_{X_0}\in C(X_0,D)$. Therefore we can define the map $\theta\in C(X_0,\mathbb{R})$ by $\theta_g(x)=\log (g(x))\ (x\in X_0)$ such that $g(x)=e^{i\theta_g(x)}\ (x\in X_0)$. We call $g\in G$ {\it liftable on $X_0$} if there exists $\theta\in C(X_0,\mathbb{R})$ such that $g(x)=e^{i\theta(x)}\ (x\in X_0)$, and denote by $G_1$ the set of all $g\in G$ which are liftable on $X_0$. Then $V\subset G_1$. Note that $G_1$ is a (possibly proper) subgroup of $G$. 
Define 
\[d(g)=\max_{x\in X_0}\theta(x)-\min_{x\in X_0}\theta(x),\ \ \ \ g=e^{i\theta}\in G_1.\]  
For $g\in G_1$ if there are two $\theta_1,\theta_2\in C(X_0,\mathbb{R})$ such that $g(x)=e^{i\theta_1(x)}=e^{i\theta_2(x)}\ (x\in X_0)$, then $X_0\ni x\mapsto \theta_1(x)-\theta_2(x)\in 2\pi \mathbb{Z}$ is constant, because $X_0$ is connected and $\theta_1,\theta_2$ are continuous. Thus $d(g)$ is independent of the choice of a lifting $\theta$. Moreover, we have $d(g_1g_2)\le d(g_1)+d(g_2)$ for all $g_1,g_2\in G_1$. 
Also observe that if $g\in V$, then $d(g)<\pi$ (use the canonical lift $\theta_g$).  This shows that $d(g)<n\pi$ for all $g\in V^n$.  Let $C=\max_{1\le i\le k}d(g_i)+n\pi$. 
Since $|X_0|\ge 2$, pick $x_0,x_1\in X_0$ with $x_0\neq x_1$. By Urysohn's Lemma, there exists $\theta\in C(X,\mathbb{R})$ with $0\le \theta(x)\le 2C\ (x\in X)$ such that $\theta(x_0)=0, \theta(x_1)=2C$. Define $g\in G_1$ by $g(x)=e^{i\theta(x)}\ (x\in X)$. Then $d(g)\ge 2C$. We show that $g\notin FV^n$. Indeed, if $g\in FV^n$, then $g=g_ih_1\cdots h_n$ for some $i\in \{1,\dots,k\}$ and $h_j\in V\ (1\le j\le n)$. Since $g,h_j\in G_1\ (1\le j\le n)$ and $G_1$ is a subgroup of $G$, we have $g_i\in G_1$ as well, so that $d(g)\le d(g_i)+\sum_{j=1}^nd(h_j)\le C$, contradicting $d(g)\ge 2C$. This shows that $g\notin FV^n$. Therefore $G\neq FV^n$. Since $F\subset G$ and $n$ are arbitrary, $G$ is not bounded.
\end{proof}
 We note the following basic facts.  
\begin{lemma}\label{lem: loc connected bounded}
Let $G$ be a locally connected topological group. Then $G$ is bounded if and only if the identity component $G_0$ is bounded and $G$ has finitely many connected components. 
\end{lemma}

\begin{proof}See \cite[Proposition 1.14]{Rosendal13} ($G$ has finitely many connected components, if and only if every open subgroup of $G$ has finite index). 
\end{proof}

\begin{corollary}\label{cor: U(A) is locally connected}
Let $A$ be a unital {\rm{C}}$^{\ast}$-algebra. Then $\mathcal{U}(A)_u$ is locally connected. Consequently, $\mathcal{U}(A)_u$ is bounded if and only if $\mathcal{U}(A)_u$ has finitely many connected components and $\mathcal{U}_0(A)_{u}$ is bounded. 
\end{corollary}
\begin{proof}
Note that if $u\in \mathcal{U}(A)$ satisfies $\sigma(u)\neq \mathbb{T}$, then $u\in \mathcal{U}_0(A)$. In particular, $V_r=\{u\in \mathcal{U}(A)\mid \|u-1\|<r\}\subset \mathcal{U}_0(A)$ for every $0<r<2$. Thus $\mathcal{U}(A)_u$ is locally connected. Then use Lemma \ref{lem: loc connected bounded} to get the latter conclusion. 
\end{proof}
Thus, in view of Corollary \ref{cor: U(A) is locally connected}, it is natural to pay attention to the identity component of the unitary group. Then we characterize a unital C$^*$-algebra for which the connected component of the unitary group is bounded. We need the notion of exponential length introduced by Ringrose (see \cite{Phillips93} for more details about exponential length and a closely related notion of exponential rank).  
\begin{definition}[\cite{Ringrose91}]\label{def: exp length} Let $A$ be a unital C$^{\ast}$-algebra. The {\it exponential length} $\text{cel}(u)$ of $u\in \mathcal{U}_0(A)$ is given as follows: 
\[\text{cel}(u)=\inf \left \{\sum_{k=1}^n\|h_k\|\middle|\,u=\exp (ih_1)\cdots \exp (ih_n),\ \ h_1,\dots,h_n\in A_{\rm{sa}}\right \}.\]
We then define the {\it exponential length} of $A$ by 
\[\text{cel}(A)=\sup \{\text{cel}(u)\,\mid\, u\in \mathcal{U}_0(A)\}.\]
\end{definition}
The exponential length can be defined for non-unital C$^*$-algebra, but we do not need it here. 
We need the following results due to Ringrose (see \cite[Corollary 2.5, Theorem 2.6 (i)(ii), Corollary 2.7, Proposition 2.9]{Ringrose91}) and the result due to Lin \cite{Lin93} (see also the work by Phillips \cite{Phillips94} for the purely infinite simple case). 
\begin{theorem}[\cite{Ringrose91}]\label{thm: ringrose exp} Let $A$ be a unital {\rm{C}}$^*$-algebra and $n\in \mathbb{N}$. 
\begin{list}{}{}
\item[{\rm{(i)}}] If $h\in A_{\rm{sa}}$ is an element with $\|h\|\le \pi$, then ${\rm{cel}}(e^{ih})=\|h\|$.
\item[{\rm{(ii)}}] If $u\in \mathcal{U}_0(A)$ satisfies ${\rm{cel}}(u)<k\pi$, then there exist $h_1,\dots,h_k\in A_{\rm{sa}}$ such that $\|h_i\|<\pi\,(1\le i\le k)$ and $u=e^{ih_1}\cdots e^{ih_k}$. In particular, ${\rm{cer}}(u)\le k$. 
\end{list}
\end{theorem}
\begin{theorem}[\cite{Lin93}]\label{thm: Lin93} Let $A$ be a {\rm{C}}$^*$-algebra of real rank zero. Then $A$ has property weak {\rm{(FU)}}: The set $\{e^{ia}; a\in \tilde{A}_{\rm{sa}}\ \text{has\ finite\ spectrum}\}$ is norm-dense in $\mathcal{U}(\tilde{A})_0$, where $\tilde{A}$ is the unitization of $A$. 
\end{theorem}
\begin{theorem}\label{thm: bounded and cel} Let $A$ be a unital ${\rm{C}}^{\ast}$-algebra. Then the following conditions are equivalent. 
\begin{list}{}{}
\item[{\rm{(i)}}] $\mathcal{U}_0(A)$ is bounded.
\item[{\rm{(ii)}}] ${\rm{cel}}(A)<\infty$. 
\end{list}
In particular, $\mathcal{U}_0(A)$ is bounded if $A$ is a unital {\rm{C}}$^*$-algebra of real rank zero. 
\end{theorem}
\begin{proof}
(ii)$\Rightarrow$(i) Fix the smallest $d\in \mathbb{N}$ such that $\text{cel}(A)<d\pi$.  Let $u\in \mathcal{U}_0(A)$. 
We show that for every $0<r<2$, there exists $n(r,d)\in \mathbb{N}$ such that $V_r^{n(r,d)}=\mathcal{U}_0(A)$. Choose $\theta(r),I_r$ as in Definition \ref{def: theta(r)}. Let $n_0(r)=\left [\frac{\pi}{\theta(r)}\right ]+1$, so that $\frac{\pi}{n_0(r)}<\theta(r)$. By Theorem \ref{thm: ringrose exp}(ii), 
there exist $h_1,\dots,h_d\in A_{\rm{sa}}$ with $\|h_j\|<\pi$ such that 
$u=e^{ih_1}\cdots e^{ih_d}$. Then $\sigma(\frac{h_k}{n_0(r)})\subset I_r$, so that $\exp (i\frac{h_k}{n_0(r)})\in V_r$ for $1\le k\le d$. Thus 
\[u=\exp (i\tfrac{h_1}{n_0(r)})^{n_0(r)}\cdots \exp (i\tfrac{h_d}{n_0(r)})^{n_0(r)}\in V_r^{dn_0(r)}.\]
Thus $n(r,d)=dn_0(r)$ works.\\ 
(i)$\Rightarrow$(ii) Assume that $\mathcal{U}_0(A)$ is bounded, and fix $0<r<2$. Then there exist a finite set $F=\{v_1,\dots,v_j\}\subset \mathcal{U}_0(A)$ and $m=m(r)\in \mathbb{N}$ such that $\mathcal{U}_0(A)=FV_r^m$. Because $\mathcal{U}_0(A)$ is connected and $V_r$ is an open neighborhood of 1 in $\mathcal{U}_0(A)$, for each $1\le i\le j$, there exists $m_j$ such that $v_j\in V_r^{m_j}$, so that $\mathcal{U}_0(A)=V_r^n$, where $n=\max_{1\le i\le j}m_j+m$. Then if $u\in \mathcal{U}_0(A)$, there exists $u_1,\dots,u_n\in V_r$ such that $u=u_n\cdots u_1$. Let $1\le k\le n$. 
Since $\sigma(u_k)\subset \{e^{i\theta}; \theta\in I_r\}$, there exists $h_k\in A_{\rm{sa}}$ with $\|h_k\|<\pi$ such that $e^{ih_k}=u_k$. By ${\rm{cel}}(vw)\le {\rm{cel}}(v)+{\rm{cel}}(w)\,(v,w\in \mathcal{U}_0(A))$ and Theorem \ref{thm: ringrose exp}(i), we obtain ${\rm{cel}}(u)<n\pi$.  
Since $u$ is arbitrary, $\text{cel}(A)\le n\pi$ holds. Finally, if $A$ moreover has real rank zero, then by Lin's Theorem \ref{thm: Lin93}, ${\rm{cel}}(A)=\pi$ holds (see also \cite[Theorem 3.5]{Phillips95}).    
\end{proof}
\begin{remark}\label{rem: boundedness for C(X)}
Let $X$ be a compact Hausdorff space, then the conditions (i) (ii) in Proposition \ref{prop: boundedness for abelian U(A)} are also equivalent to the condition
\begin{list}{}{}
\item[{\rm{(ii)}}] $\mathcal{U}_0(C(X))_u$ is bounded.
\end{list}
Indeed,  (i)$\Rightarrow$(iii) follows from Corollary \ref{cor: U(A) is locally connected}. For (iii)$\Rightarrow$(i),  $\mathcal{U}_0(C(X))$ is bounded, then by Theorem \ref{thm: bounded and cel}, ${\rm{cel}}(C(X))<\infty$ holds. Then by \cite[Corollary 3.3]{Phillips95}, $X$ is totally disconnected. In this case, $C(X)$ is a unital AF algebra and therefore $\mathcal{U}(C(X))=\mathcal{U}_0(C(X))$ is bounded. 
\end{remark}
\begin{example}\label{rem: real rank zero is not necessary} Although $\mathcal{U}_0(C(\mathbb{T}))_u$ is not bounded (Example \ref{ex: C(T) is not bounded} and Remark \ref{rem: boundedness for C(X)}), Phillips \cite[Theorem 6.5 (2)]{Phillips93} (cf. \cite{Phillips92}) showed that for $A=C(\mathbb{T})\otimes \mathbb{B}(\ell^2)$, ${\rm{cel
}}(A)=2$. Therefore $\mathcal{U}_0(A)_u$ is bounded by Theorem \ref{thm: bounded and cel}. Note that $A$ does not have real rank zero (a more general result can be found in Kodaka--Osaka's work \cite{KodakaOsaka95}). This shows that the real rank zero property is not a necessary property for the identity component of the unitary group of a unital C$^*$-algebra to be bounded.   
\end{example}
Finally we make another remark about the relationship between the property (OB) introduced and studied by Rosendal \cite{Rosendal09} and boundedness in the sense of Hejcman \cite{Hejcman58} for groups of the form $\mathcal{U}(A)$. 
In general, boundedness implies property (OB), but the converse may fail. However, if $A$ has real rank zero (here, $A$ is said to have real rank zero, if invertible self-adjoint elements are dense in $A_{\rm{sa}}$), then they are equivalent (cf. Remark \ref{rem: real rank zero is not necessary}):
\begin{corollary}\label{cor: characterization of bounded for RR=0}
Let $A$ be a unital separable {\rm{C}}$^*$-algebra of real rank zero. Then the following three conditions are equivalent. 
\begin{list}{}{}
\item[{\rm{(i)}}] $\mathcal{U}(A)_u$ is bounded. 
\item[{\rm{(ii)}}] $\mathcal{U}(A)_u$ has property {\rm{(OB)}}. 
\item[{\rm{(iii)}}] $\mathcal{U}(A)_u$ has finitely many connected components. 
\end{list}
\end{corollary}

\begin{proof}
(i)$\Rightarrow$(ii) is clear.\\
(ii)$\Rightarrow$(iii) Assume (iii) does not hold, so $\Gamma=\mathcal{U}(A)/\mathcal{U}_0(A)$ is a countable infinite discrete group ($\mathcal{U}(A)_{0,u}$ is an open subgroup because $\mathcal{U}(A)_u$ is locally connected). Let $q\colon \mathcal{U}(A)\to \Gamma$ be the quotient map. 
If $\Gamma$ has property (T), then $\Gamma$ is finitely generated. Then $\Gamma$ acts by left multiplication on its Cayley graph with respect to a finite generating set, which has unbounded orbit. If $\Gamma$ does not have property (T), then it has some affine isometric action on a real Hilbert space with an unbounded orbit. In either case, we have an isometric action $\alpha\colon \Gamma\curvearrowright X$ on a metric space with an unbounded orbit. Then $\beta(u)=\alpha(q(u))\ (u\in \mathcal{U}(A))$ defines an isometric action $\beta\colon \mathcal{U}(A)\curvearrowright X$ with an unbounded orbit. Thus $\mathcal{U}(A)$ does not have property (OB).\\
(iii)$\Rightarrow$(i) By Theorem \ref{thm: bounded and cel}, $\mathcal{U}_0(A)$ is bounded. Then (i) follows from Lemma \ref{lem: loc connected bounded}. 
\end{proof}
\subsection{Property (FH)}
In \cite[Example 6.7]{Pestov18}, Pestov shows that the group $\mathcal{U}_2(\ell^2)$ does not have property (T), and it was left unanswered (see \cite[$\S$3.5]{Pestov18}) whether $\mathcal{U}_2(\ell^2)$ has property (FH). In this section, we show that it does not have property (FH). Actually we show stronger results.  
Let $M$ be a properly infinite von Neumann algebra with a normal faithful semifinite trace $\tau$. We will show that the group $\mathcal{U}_p(M,\tau)$ ($1\le p<\infty$) does not have property (FH), hence they do not have property (T) either. 
We need two different proofs depending on whether $1\le p\le 2$ or $2<p<\infty$. 
\subsection{Case $1\le p\le 2$}
For a projection $e\in M$, we set
\[
M_e = \{ex|_{{\rm ran}\,(e)} \mid x\in M\}.
\]
It is known that  $M_e$ is a von Neumann algebra acting on the Hilbert space ${\rm ran}\,(e)$. We say that $e$ is $\tau$-finite, if $\tau(e)<\infty$. Note that in general  the $\tau$-finiteness implies that $e$ is a finite projection, but the converse need not hold. 

\begin{theorem}\label{not FH new}
Let $M$ be a semifinite von Neumann algebra with separable predual and $\tau$ be a faithful normal semifinite trace on $M$. 
Let $1\leq p\leq 2$, and let $G$ be a closed subgroup of $\mathcal{U}_p(M,\tau)$ with property {\rm{(FH)}}.
Then there exists a $\tau$-finite projection $e$ in $M$ such that $e$ reduces all elements in $G$,
that $u=1$ on ${\rm ran}\,(e)^{\perp}$, and that
the map
\[
G\to\mathcal{U}(M_e),\ \ \ \ \ u\mapsto u|_{{\rm ran}\,(e)}
\] 
defines an isomorphism from $G$ onto a closed subgroup of $\mathcal{U}(M_e)$.
In particular, $G$ is a Polish group of finite type.
\end{theorem}
We need the following well-known Lemma.  
\begin{lemma}\label{lem: p norms agree with SOT}Let $M$ be a finite von Neumann algebra with a normal faithful tracial state $\tau$. Then on bounded subsets of $M$, the topology given by the $p$-norm $\|\cdot\|_p$ agrees with the {\rm{SOT}} for any $1\le p<\infty$. 
\end{lemma}
\begin{proof}
It is well-known that $\|\cdot \|_2$ agrees with the SOT on bounded subsets of $M$. 
Let $(x_n)_{n=1}^{\infty}$ be a bounded sequence in $M$ converging to 0 in SOT. 
Then because $M$ is of finite type, the $*$-operation is SOT-continuous on bounded subsets of $M$. Thus, $x_n^*x_n\stackrel{n\to \infty}{\to}0$ (SOT). Hence (use Stone--Weierstrass Theorem) $|x_n|^{\frac{p}{2}}\stackrel{n\to \infty}{\to}0$ (SOT), whence $\||x_n|^{\frac{p}{2}}\|_2^2=\tau(|x_n|^p)=\|x_n\|_p^p\stackrel{n\to \infty}{\to}0$. Converesely, assume that $\|x_n\|_p^p=\||x_n|^{\frac{p}{2}}\|_2^2\stackrel{n\to \infty}{\to}0$. Then $|x_n|^{\frac{p}{2}}\stackrel{n\to \infty}{\to}0$ (SOT), whence 
$|x_n|=(|x_n|^{\frac{p}{2}})^{\frac{2}{p}}\stackrel{n\to \infty}{\to}0$ (SOT). Therefore 
$\|x_n\|_2=\||x_n|\|_2\stackrel{n\to \infty}{\to}0$, so that $x_n\stackrel{n\to \infty}{\to}0$ (SOT). 
\end{proof}
\begin{proof}[Proof of Theorem \ref{not FH new}]
We regard $L^2(M,\tau)$ as a real Hilbert space by taking the real part of the inner product.
Define an affine isometric action of $G$ on $L^2(M,\tau)$ by
\[
u\cdot x = ux + (u-1), \ \ \ \ \ u\in G,\ x\in L^2(M,\tau).
\]
This action is well-defined because the inequality
\[
\tau(|u-1|^2)=\tau(|u-1|^{2-p}\cdot|u-1|^p)\leq\left\||u-1|^{2-p}\right\|_{\infty}\tau(|u-1|^p)<\infty
\]
implies $u-1\in L^2(M,\tau)$.
Note that the above argument shows that we have
\[
\|u-1\|_2 \leq \left\||u-1|^{2-p}\right\|_{\infty}^{1/2}\|u-1\|_p^{p/2}
\leq 2^{1-p/2}\|u-1\|_p^{p/2},\ \ \ \ \ u\in G.
\]

To prove the continuity of the action, let $G\ni u_n\stackrel{n\to \infty}{\to}u\in G$ and $L^2(M,\tau)\ni x_n\stackrel{n\to \infty}{\to}x \in L^2(M,\tau)$.
We see that
\begin{align*}
\|u_n\cdot x_n-u\cdot x\|_2 &= \|[u_nx_n+(u_n-1)]-[ux+(u-1)]\|_2\\ 
&\leq \|u_nx_n-ux\|_2 + \|u_n-u\|_2\\
&\leq \|u_nx_n-u_nx\|_2 +\|u_nx-ux\|_2 + \|u_n-u\|_2\\
&\leq \|x_n-x\|_2 + \|u_n-u\|_2\|x\|_{\infty} +\|u_n-u\|_2\\
&\leq \|x_n-x\|_2 + 2^{1-p/2}\|u_n-u\|_p^{p/2}\|x\|_{\infty} +2^{1-p/2}\|u_n-u\|_p^{p/2}\\
&\stackrel{n\to \infty}{\to} 0,
\end{align*}
whence the action is continuous.

Since $G$ has property (FH), there is a fixed point $x_0\in L^2(M,\tau)$.
Then we have $(u-1)(x_0+1)=0$ for any $u\in G$.
Thus the projection $e$ of $L^2(M,\tau)$ onto $\ker{(x_0^*+1)}$ reduces all elements in $G$.
We show that $e$ is a $\tau$-finite projection in $M$.
Since $\ker{(x_0^*+1)}=\ker{(|x_0^*+1|)}$, $e$ is in $M$.
It is clear that $x_0^*e=-e$, thus $ex_0=-e$,
which means that $ex_0x_0^*e=e$.
Then
\[
\tau(e) = \tau(ex_0x_0^*e) = \tau(x_0^*ex_0) \leq \tau(x_0^*x_0) <\infty,
\]
whence $e$ is a $\tau$-finite projection.

The map
\[
G\to\mathcal{U}(M_e)_s,\ \ \ \ \ u\mapsto u|_{{\rm ran}\,(e)}
\]
is a topological group isomorphism onto a subgroup of $\mathcal{U}(eMe)$. Here, we use the fact that 
for $u\in G$, the identity $\|u-1\|_p=\|(u-1)e\|_p$ holds, and that on $\mathcal{U}(eMe)$, the $p$-norm agrees with the SOT by Lemma \ref{lem: p norms agree with SOT}.  
Since $G$ and $\mathcal{U}(eMe)$ are Polish, the image of the map is closed.
This finishes the proof.
\end{proof}

\begin{corollary}
Let $1\leq p\leq 2$.
Then $\mathcal{U}_p(M,\tau)$ has property {\rm{(FH)}} if and only if 
$M$ is a finite von Neumann algebra.
\end{corollary}

\begin{proof}
If $M$ is a finite von Neumann algebra, then by Lemma \ref{lem: bounded for vNa}, $\mathcal{U}_p(M,\tau)=\mathcal{U}(M)_s$  is bounded, whence it has property (FH).
Conversely, assume that $\mathcal{U}_p(M,\tau)$ has property (FH).
We apply Theorem \ref{not FH new} to $G=\mathcal{U}_p(M,\tau)$ to find a $\tau$-finite projection $e\in M$ such that $\mathcal{U}_p(M,\tau)\ni u\mapsto u|_{\rm{ran}}(e)\in \mathcal{U}(M_e)_s$ is an isomorphism onto its image. 
Since $\mathcal{U}_p(M,\tau)$ generates $M$, the projection $e$ commutes with
all elements in $M$. 
Thus $M$ is a von Neumann subalgebra of a finite von Neumann algebra $M_e\oplus \mathbb{C}e^{\perp}$.
Hence $M$ is a finite von Neumann algebra as well.
\end{proof}

\begin{corollary}\label{cor: closed subgroup of U_pl2 is compact}
Let $1\leq p\leq 2$, and let $G$ be a closed subgroup of $\mathcal{U}_p(\ell^2)$ with property {\rm{(FH)}}.
Then $G$ is a compact Lie group.
\end{corollary}

\begin{proof}
Since every finite projection $e$ in $\mathbb{B}(\ell^2)$ is of finite rank,
the group $\mathcal{U}(\mathbb{B}(\ell^2)_e)$ 
is isomorphic to the unitary group of degree ${\rm dim}\,{\rm ran}\,{(e)}$. Then by Theorem \ref{not FH new}, $G$ is isomorphic to a closed subgroup of a finite-dimensional unitary group, so it is a compact Lie group. 
\end{proof}
\subsection{Case $2<p<\infty$}
For the case $2<p<\infty$, we give two different proofs of the fact that $\mathcal{U}_p(M,\tau)$ does not have property (FH) if $M$ is a properly infinite semifinite von Neumann algebra with separable predual. The first proof (Theorem \ref{thm no FH for big G}), though the argument is somewhat involved, provides a byproduct that for the unitary $u\in \mathcal{U}(M)$, a noncommutative analogue of Leach's result \cite{Leach56} on a converse of the H\"older's inequality holds with some additional estimate on the element $z$ paired with $u-1$ by the duality. See Proposition \ref{prop: unbounded L2}. Hopefully this and some of the other lemmas used in the first proof might be useful in other purposes. The second proof (Theorem \ref{thm: (FH) implies Ufin for p>2}) is simpler and in a sense a refinement of the argument in Theorem \ref{not FH new}.    
Below, $*$SOT stands for the $*$-strong operator topology. 
\begin{theorem}\label{thm no FH for big G}
Let $2<p<\infty$. Let $M$ be a properly infinite semifinite von Neuman algebra with separable predual and $\tau$ be a normal faithful semifinite trace on $M$. Let $G$ be a closed subgroup of $\mathcal{U}_p(M,\tau)$ such that 
$\overline{G}^{\rm{*SOT}}\not\subset \mathcal{U}_p(M,\tau)$. Then $G$ does not have property {\rm{(FH)}}. In particular, $\mathcal{U}_p(M,\tau)$ does not have property {\rm{(FH)}}.   
\end{theorem}
\begin{lemma}\label{lem: continuity of Lpaction}
Let $2<p<\infty$ and $M$ be a semifinite von Neuman algebra with separable predual and $\tau$ a normal faithful semifinite trace on $M$. Let $u_n,u\in \mathcal{U}_p(M,\tau)\ (n\in \mathbb{N})$ be such that $\lim_{n\to \infty}\|u_n-u\|_p=0$. Then 
for every $\xi\in L^2(M,\tau)$, $\lim_{n\to \infty}\|u_n\xi-u\xi\|_2=0$ holds. 
\end{lemma}
\begin{proof}
Since $\sup_{n\in \mathbb{N}}\|u_n-u\|_{\infty}<\infty$ and $M\cap L^2(M,\tau)$ is $\|\cdot\|_2$-dense in $L^2(M,\tau)$, it suffices to show that 
$\lim_{n\to \infty}\|(u_n-u)x\|_2=0$ for every $x\in M\cap L^2(M,\tau)$. Choose $0<r<2$ such that $\tfrac{1}{2}+\tfrac{1}{p}=\tfrac{1}{r}$. Then by the generalized noncommutative H\"older's inequality, 
\eqa{
\|(u_n-u)x\|_2^2&=\tau(|(u_n-u)x|^{\frac{r}{2}}|(u_n-u)x|^{2-r}|(u_n-u)x|^{\frac{r}{2}})\\
&\le \|(u_n-u)x\|_{\infty}^{2-r}\|(u_n-u)x\|_r^r\\
&\le 2^{2-r}\|x\|_{\infty}^{2-r}\|u_n-u\|_p^r\|x\|_2^r\stackrel{n\to \infty}{\to}0.
} 
This finishes the proof. 
\end{proof}

\begin{proposition}\label{prop: unbounded L2}
Let $2<p<\infty$, $M$ be a properly infinite semifinite von Neumann algebra with separable predual, $\tau$ be a faithful normal semifinite trace on $M$ and let $u\in \mathcal{U}(M)\setminus \mathcal{U}_p(M,\tau)$. Let $2<q<\infty$ be such that $\tfrac{1}{p}+\tfrac{1}{q}=\tfrac{1}{2}$. Then there exists $z\in L^q(M,\tau)_+$ such that $1_{[\varepsilon,\infty)}(z)z\in L^2(M,\tau)$ for every $\varepsilon>0$ and $(u-1)z\notin L^2(M,\tau)$. 
\end{proposition}

We need the following classical result due to Leach. 
\begin{theorem}[\cite{Leach56}]\label{thm: Leach}
Let $1<p,q<\infty$ be such that $\frac{1}{p}+\frac{1}{q}=1$. 
Let $(X,\mu)$ be a measure space. The following two conditions are equivalent. 
\begin{list}{}{}
\item[{\rm{(i)}}] For every measurable subset $E\subset X$ with $\mu(E)=\infty$, there exists a measurable subset $F\subset E$ with $0<\mu(F)<\infty$. 
\item[{\rm{(ii)}}] The converse of the H\"older's inequality holds, i.e., whenever a measurable function $f\colon X\to \mathbb{C}$ has the property that $fg\in L^1(X,\mu)$ for every $g\in L^q(X,\mu)$, then $f\in L^p(X,\mu)$ holds. 
\end{list}
\end{theorem}

\begin{proof}[Proof of Proposition \ref{prop: unbounded L2}]
By assumption, $M$ is of the form 
\begin{equation}
M=\bigoplus_{j\in J}Q_j\oplus Q_0\label{eq: direct sum dec}
\end{equation} where $J\subset \mathbb{N}$ is a subset and each $Q_j\,(j\in \mathbb{N})$ is a type I$_{\infty}$ factor (if $J\neq \emptyset$. The case $J=\emptyset$ is allowed, in which case (\ref{eq: direct sum dec}) means $M=Q_0$ is diffuse) and $Q_0$ is either $\{0\}$ (in which case (\ref{eq: direct sum dec}) means $J\neq \emptyset$ and $M=\bigoplus_{j\in J}Q_j$ is completely atomic) or a diffuse properly infinite semifinite von Neumann algebra. We treat the case $J\neq \emptyset$ and $Q_0\neq \{0\}$. The other cases are easier to handle. Set $J_0=J\cup \{0\}$ and we write $x\in M$ as $x=(x_j)_{j\in J_0}$, where $x_j\in Q_j\,(j\in J_0)$, so that $\|x\|_p^p=\sum_{j\in J_0}\|x_j\|_p^p$ (whether or not the both sides are finite). 

Let $u=\int_{\mathbb{T}}\lambda\,{\rm{d}}e(\lambda)$ be the spectral resolution of $u$.
Then for each $j\in J_0$, $u_j=\int_{\mathbb{T}}\lambda\,{\rm{d}}e_j(\lambda)$ is the spectral resolution of $u_j$, where $u=(u_j)_{j\in J_0}$ and $e_j(\cdot)=1_{Q_{j}}e(\cdot)=e(\cdot)1_{Q_j}$. 
Let $\lambda_{\infty}=\lambda_{\infty+1}=e^{i\frac{\pi}{2}}$. 
For each $n\in \mathbb{N}$, define $\lambda_n=e^{i\theta_n},\,\theta_n\in (0,\tfrac{\pi}{2}]$ so that $|\lambda_n-1|=2^{-n+\tfrac{3}{2}}$ holds. 
Let $A_{\infty}=\{\lambda \in \mathbb{T}\mid |\lambda-1|\ge 2^{\frac{1}{2}}\}$ and 
$A_n=\{e^{i\theta}\mid \theta_{n+1}\le |\theta|<\theta_{n}\}\,(n\in \mathbb{N})$. 
Also let $p_n=e(A_n),\,p_{j,n}=e_j(A_n)$ for each $n\in \mathbb{N}\cup \{\infty\}$ and $j\in J_0$. Note that $\mathbb{T}=\{1\}\sqcup \bigsqcup_{k\in \mathbb{N}\cup \{\infty\}}A_k$ is a partition of $\mathbb{T}$.\\

By $u-1\notin L^p(M,\tau)$, we have $\sum_{j\in J_0}\|u_j-1_{Q_j}\|_p^p=\infty$. 
For each $j\in J$, define $c_j=\tau(q_j)$, where $q_j$ is a minimal projection in the type I$_{\infty}$ factor $Q_j$, which is independent of the choice of $q_j$. We also set $c_0=1$.  
For each $\tilde{J}\subset J_0$, we define $c_{\tilde{J}}=\inf_{j\in \tilde{J}}c_j$.\\ \\ 
For each $j\in J$, the spectral resolution of $u_j$ takes the form $u_j=\sum_{k=1}^{\infty}\lambda_{j}(k)e_j(k)$, where $\lambda_j(k)\in \mathbb{T}\,(k\in \mathbb{N})$ and $(e_j(k))_{k=1}^{\infty}$ is a sequence of mutually orthogonal minimal projections in $Q_j$ with $\sum_{k=1}^{\infty}e_j(k)=1$. 
We further divide the situation into two cases.\\ \\ 
\textbf{Case 1.} There exists $\tilde{J}\subset J$ such that $c_{\tilde{J}}>0$ and $\sum_{j\in \tilde{J}}\|u_j-1_{Q_j}\|_p^p=\infty$.\\
Define a $\sigma$-finite measure space $(X,\mu)=(\tilde{J}\times \mathbb{N}, \sum_{(j,k)\in X}c_j\delta_{j,k})$, which satisfies the condition (i) of Theorem \ref{thm: Leach}. Define $f\colon X\to \mathbb{C}$ by $f(j,k)=|\lambda_j(k)-1|^2\,((j,k)\in X)$. Then 
$$\|f\|_{\frac{p}{2}}^{\frac{p}{2}}=\sum_{(j,k)\in X}|\lambda_j(k)-1|^pc_j=\sum_{j\in \tilde{J}}\|u_j-1_{Q_j}\|_p^p=\infty,$$
and $f\notin L^{\frac{p}{2}}(X,\mu)$ holds. 
By Theorem \ref{thm: Leach}, there exists $g\in L^{\frac{q}{2}}(X,\mu)$ such that $fg\notin L^1(X,\mu)$. Define 
$$z=\sum_{(j,k)\in X}|g(j,k)|^{\frac{1}{2}}e_j(k).$$
Then 
$$\|z\|_q^q=\sum_{(j,k)\in X}|g(j,k)|^{\frac{q}{2}}c_j=\|g\|_{\frac{q}{2}}^{\frac{q}{2}}<\infty.$$
Thus $z\in L^q(M,\tau)_+$. Let $\varepsilon>0$ and $X_{\varepsilon}=\{(j,k)\in X\mid\,|g(j,k)|^{\frac{1}{2}}\ge \varepsilon\}$. Then the following inequality holds. 
\[\sum_{(j,k)\in X_{\varepsilon}}\varepsilon^qc_{\tilde{J}}\le \sum_{(j,k)\in X_{\varepsilon}}|g(j,k)|^{\frac{q}{2}}c_j\le \|g\|_{\frac{q}{2}}^{\frac{q}{2}}<\infty.\]
By $c_{\tilde{J}}>0$, this implies that $X_{\varepsilon}$ is a finite set, so that 
\[\|1_{[\varepsilon,\infty)}(z)z\|_2^2=\sum_{(j,k)\in X_{\varepsilon}}|g(j,k)|c_j<\infty.\]
Therefore $1_{[\varepsilon,\infty)}(z)z\in L^2(M,\tau)$ holds. Moreover, 
\[\|(u-1)z\|_2^2=\sum_{(j,k)\in X}|\lambda_j(k)-1|^2\,|g(j,k)|c_j=\|fg\|_1=\infty,\]
whence $(u-1)z\notin L^2(M,\tau)$ holds.\\ \\
\textbf{Case 2.} There is no $\tilde{J}\subset J$ satisfying the condition stated in Case 1.\\
In this case, for any $\tilde{J}\subset J$, the implication $c_{\tilde{J}}>0\Rightarrow \sum_{j\in \tilde{J}}\|u_j-1_{Q_j}\|_p^p<\infty$ holds. Let $\tilde{J}=\{j\in J\mid c_j\le 1\}$ and $\tilde{J}_0=\tilde{J}\cup \{0\}$. Then obviously $c_{J\setminus \tilde{J}}\ge 1>0$, so that $\sum_{j\in J\setminus \tilde{J}}\|u_j-1_{Q_j}\|_p^p<\infty$. 
By $\sum_{j\in J_0}\|u_j-1_{Q_j}\|_p^p=\infty$, we also have $\sum_{j\in \tilde{J}_0}\|u_j-1_{Q_j}\|_p^p=\infty$.\\

Let $\tilde{p}_k=\sum_{j\in \tilde{J}_0}p_{j,k}\,(k\in \mathbb{N}\cup \{\infty\})$. 
Then by $|\lambda-1|\le |\lambda_k-1|,(\lambda\in A_k,\,k\in \mathbb{N})$ and $|\lambda-1|\le 2\,(\lambda\in A_{\infty})$, we have the following inequality:
\[\infty=\sum_{j\in \tilde{J}_0}\|u_j-1_{Q_j}\|_p^p\le 2^{\frac{p}{2}}\tau(\tilde{p}_{\infty})+\sum_{k=1}^{\infty}|\lambda_k-1|^p\tau(\tilde{p}_k),\]
we need to consider two cases.\\
\textbf{Case 2-1.} $\tau(\tilde{p}_k)=\infty$ for some $k\in \mathbb{N}\cup \{\infty\}$.\\
Assume first that $\tau(p_{j,k})<\infty$ for all $j\in \tilde{J}_0$. 
In this case, we may find a sequence $(e_n)_{n=1}^{\infty}$ of mutually orthogonal projections in $M$, such that for each $n\in \mathbb{N}$, $e_n$ is of the form $e_n=\sum_{j\in \tilde{J_n}}f_{j,n}$, where $(\tilde{J}_n)_{n=1}^{\infty}$ is a sequence of (possibly non-disjoint) subsets of $\tilde{J}_0$, and  for each $j\in \tilde{J}_n$, $f_{j,n}$ is a subprojection of $p_{j,k}\in Q_j$ with $1\le \tau(e_n)\le 2$ (the upper bound $\tau(e_n)\le 2$ can be made possible thanks to the condition that $c_j\le 1\, (j\in \tilde{J})$ and the diffuseness of $Q_0$). Define 
\[z=\sum_{n=1}^{\infty}n^{-\frac{1}{2}}e_n.\]
Then $\|z\|_q^q=\sum_{n=1}^{\infty}n^{-\frac{q}{2}}\tau(e_n)<\infty$ by $q>2$ and $\tau(e_n)\le 2\,(n\in \mathbb{N})$. Therefore $z\in L^q(M,\tau)_+$. For each  $\varepsilon>0$, 
\[\|1_{[\varepsilon,\infty)}(z)z\|_2^2=\sum_{n\le \varepsilon^{-2}}n^{-1}\tau(e_n)<\infty\]
because the right hand side is a finite sum. Thus $1_{[\varepsilon,\infty)}(z)z\in L^2(M,\tau)$ holds. Moreover, thanks to $|\lambda-1|\ge |\lambda_{k+1}-1|\,(\lambda\in A_k)$ and $\tau(e_n)\ge 1\,(n\in \mathbb{N})$, we see that
\eqa{
\|(u-1)z\|_2^2&=\left \|\sum_{j\in J_0}(u_j-1_{Q_j})\cdot \sum_{n=1}^{\infty}n^{-\frac{1}{2}}\sum_{j\in \tilde{J}_n}f_{j,n}\right \|_2^2\\
&=\sum_{n=1}^{\infty}n^{-1}\sum_{j\in \tilde{J}_n}\|(u_j-1_{Q_j})f_{j,n}\|_2^2\\
&=\sum_{n=1}^{\infty}n^{-1}\sum_{j\in \tilde{J}_n}\int_{A_k}|\lambda-1|^2\,{\rm{d}}\tau(e_j(\lambda)f_{j,n})\\
&\ge |\lambda_{k+1}-1|^2\sum_{n=1}^{\infty}n^{-1}\tau(e_n)=\infty.
}
Therefore $(u-1)z\notin L^2(M,\tau)$.\\
Next, assume that $\tau(p_{j,k})=\infty$ for some $j\in \tilde{J}_0$. 
Then by using the assumption that $Q_j$ is either a type I$_{\infty}$ factor or diffuse, we can find a sequence $(e_n)_{n=1}^{\infty}$ of mutually orthogonal projections in $Q_j$ with $\tau(e_n)=c_j\,(n\in \mathbb{N})$, such that $\sum_{n=1}^{\infty}e_n=p_{j,k}$. Set $z=\sum_{n=1}^{\infty}n^{-\frac{1}{2}}e_n$. 
Then $\|z\|_q^q=\sum_{n=1}^{\infty}n^{-\frac{q}{2}}c_j<\infty$, whence $z\in L^q(M,\tau)_+$. By a similar argument as above, we have $1_{[\varepsilon,\infty)}(z)z\in L^2(M,\tau)$ and 
\[\|(u-1)z\|_2^2\ge |\lambda_{k+1}-1|^2\sum_{n=1}^{\infty}n^{-1}c_j=\infty.\]
Thus $(u-1)z\notin L^2(M,\tau)$ holds.\\ 
\textbf{Case 2-2} $\tau(\tilde{p}_k)<\infty$ for all $k\in \mathbb{N}\cup \{\infty\}$.\\
In this case, $\sum_{k=1}^{\infty}|\lambda_k-1|^p\tau(\tilde{p}_k)=\infty$, hence $\sum_{k=1}^{\infty}|\lambda_{k+1}-1|^p\tau(\tilde{p}_k)=\infty$ holds by $|\lambda_{k+1}-1|=2^{-1}|\lambda_k-1|\,(k\in \mathbb{N})$. 
In view of $|\lambda_{k}-1|=2^{-k+\frac{3}{2}}\,(k\in \mathbb{N})$, we have 
\[\sum_{k: \tau(\tilde{p}_k)<1}|\lambda_{k+1}-1|^p\tau(\tilde{p}_k)<\infty.\]
This implies that the set $I=\{k\in \mathbb{N}\mid 1\le \tau(\tilde{p}_k)(<\infty)\}$ is infinite. Let $k_1<k_2<\cdots$ be the enumeration of $I$. 
Consider the $\sigma$-finite measure space $(X,\mu)=(\mathbb{N},\,\sum_{j=1}^{\infty}a_j\delta_j)$, where $a_j=\tau(\tilde{p}_{k_j})\,(j\in \mathbb{N})$, which satisfies the condition (i) of Theorem \ref{thm: Leach}. Define $f\colon X\to \mathbb{C}$ by $f(j)=|\lambda_{k_j+1}-1|^2\,(j\in \mathbb{N})$. Then by assumption, $f\notin L^{\frac{p}{2}}(X,\mu)$. Then by Theorem \ref{thm: Leach}, there exists $g\in L^{\frac{q}{2}}(X,\mu)$ such that $fg\notin L^1(X,\mu)$. Define 
$$z=\sum_{j=1}^{\infty}|g(j)|^{\frac{1}{2}}\tilde{p}_{k_j}.$$
Since $\|z\|_q^q=\int_X|g|^{\frac{q}{2}}\,{\rm{d}}\mu<\infty,\,z\in L^q(M,\tau)_+$ holds.
For each $\varepsilon>0$, consider $J_{\varepsilon}=\{j\in \mathbb{N}\mid |g(j)|^{\frac{1}{2}}\ge \varepsilon\}$. Then by $z\in L^q(M,\tau)$, we have (use $a_j\ge 1\,(j\in \mathbb{N})$)
$$\sum_{j\in J_{\varepsilon}}\varepsilon^q\le \sum_{j\in J_{\varepsilon}}|g(j)|^{\frac{q}{2}}a_j\le \sum_{j=1}^{\infty}|g(j)|^{\frac{q}{2}}a_j=\|z\|_q^q<\infty.$$
Thus $J_{\varepsilon}$ is finite and consequently $\|1_{[\varepsilon,\infty)}(z)z\|_2^2=\sum_{j\in J_{\varepsilon}}|g(j)|a_j<\infty$. Moreover, by $|\lambda-1|\ge |\lambda_{k+1}-1|\,(\lambda\in A_k,\,k\in \mathbb{N})$, we obtain 
$$\|(u-1)z\|_2^2\ge \sum_{j=1}^{\infty}f(j)|g(j)|a_j=\infty$$
by $fg\notin L^1(X,\mu)$. Therefore $(u-1)z\notin L^2(M,\tau)$ holds. 
\end{proof}
\begin{lemma}\label{lem: not in L2}
Let $2<p<\infty$ and $M$ be a semifinite von Neumann algebra with a faithful normal semifinite trace $\tau$. Let $(u_n)_{n=1}^{\infty}$ be a sequence in $\mathcal{U}_p(M,\tau)$ converging to $u\in \mathcal{U}(M)$ in {\rm{SOT}}. Let $z$ be a positive self-adjoint operator on $L^2(M,\tau)$ affiliated with $M$, such that $1_{[\varepsilon,\infty)}(z)z\in L^2(M,\tau)$ for every $\varepsilon>0$. If $\sup_{n\in \mathbb{N}}\|(u_n-1)z\|_2<\infty$, then $(u-1)z\in L^2(M,\tau)$ holds.  
\end{lemma}
\begin{proof}
By assumption, the sequence $((u_n-1)z)_{n=1}^{\infty}$ is a bounded sequence in the separable Hilbert space $L^2(M,\tau)$, so there exists a subsequence $((u_{n_k}-1)z)_{k=1}^{\infty}$ converging weakly to some $y\in L^2(M,\tau)$. Let $z=\int_0^{\infty}\lambda\,{\rm{d}}e(\lambda)$ be the spectral resolution of $z$. For each $m\in \mathbb{N}$, define $e_m=1_{[\frac{1}{m},\infty)}(z)\in M$. 
By assumption, $e_mz=ze_m\in L^2(M,\tau)$ holds. Let $x\in L^2(M,\tau)$. Then for each $k,m\in \mathbb{N}$, 
$u_{n_k}\stackrel{k\to \infty}{\to} u$ (SOT) implies that  
\[|\nai{(u_{n_k}-1)e_mz}{x}|\stackrel{k\to \infty}{\to}|\nai{(u-1)e_mz}{x}|.\]
On the other hand, by using the trace property, we also have 
\eqa{
|\nai{(u_{n_k}-1)e_mz}{x}|&=|\nai{(u_{n_k}-1)ze_m}{x}|=|\nai{(u_{n_k}-1)z}{xe_m}|\\
&\stackrel{k\to \infty}{\to}|\nai{y}{xe_m}|\\
&\le \|y\|_2\|xe_m\|_2\le \|y\|_2\|x\|_2.
}
This shows that 
$$|\nai{(u-1)e_mz}{x}|\le \|y\|_2\|x\|_2,\ \ m\in \mathbb{N}.$$
Since $x\in L^2(M,\tau)$ is arbitrary, this implies that $\|(u-1)e_mz\|_2=\|(u-1)ze_m\|_2\le \|y\|_2.$ for every $m\in \mathbb{N}$. By the normality of the trace and $|(u-1)z|e_m|(u-1)z|\nearrow |(u-1)z|^2\,(m\to \infty)$ (SOT), we obtain 
\eqa{
\|(u-1)ze_m\|_2^2&=\tau(e_m|(u-1)z|^2e_m)=\tau(|(u-1)z|e_m|(u-1)z|)\\
&\nearrow \tau(|(u-1)z|^2)\,(m\to \infty).
}
Thus $\|(u-1)z\|_2\le \|y\|_2$ holds. This shows that $(u-1)z\in L^2(M,\tau)$. 
\end{proof}
\begin{proof}[Proof of Theorem \ref{thm no FH for big G}] 
Take $q>2$ such that $\frac{1}{p}+\frac{1}{q}=\frac{1}{2}$. By Proposition \ref{prop: unbounded L2}, there exists $z\in L^q(M,\tau)_+$ such that $1_{[\varepsilon,\infty)}(z)z\in L^2(M,\tau)$ for every $\varepsilon>0$ and $(u-1)z\notin L^2(M,\tau)$ holds. 

Define an affine isometric action of $G$ on $L^2(M,\tau)$ by 
\begin{equation}
u\cdot x=ux+(u-1)z,\ \ \ \ u\in G,\, x\in L^2(M,\tau)\label{eq: affine action}
\end{equation}
Thanks to the noncommutative H\"older's inequality, the action is well-defined  (by $(u-1)z\in L^2(M,\tau)$). We show that the action is continuous. 
Let $(v_n)_{n=1}^{\infty}$ be a sequence in $G$ converging to $v\in G$ and $(x_n)_{n=1}^{\infty}$ be a sequence in $L^2(M,\tau)$ converging to $x\in L^2(M,\tau)$. Then by Lemma \ref{lem: continuity of Lpaction}, we have 
\eqa{
\|v_n\cdot x_n-v\cdot x\|_2&=\|v_nx_n+(v_n-1)z-vx-(v-1)z\|_2\\
&\le \|v_n(x_n-x)\|_2+\|(v_n-v)x\|_2+\|(v_n-v)z\|_2\\
&\le \|x_n-x\|_2+\|(v_n-v)x\|_2+\|v_n-v\|_p\|z\|_q\\
&\stackrel{n\to \infty}{\to}0.
}
This shows that the action is continuous. If $G$ had property (FH), there would exist a $G$-fixed point $x_0\in L^2(M,\tau)$. By $\overline{G}^{\rm{*SOT}}\not\subset \mathcal{U}_p(M,\tau)$, there exists $u\in \mathcal{U}(M)\setminus \mathcal{U}_p(M,\tau)$ and a sequence $(u_n)_{n=1}^{\infty}$ in $G$ such that $u_n\stackrel{n\to \infty}{\to}u$ ($*$SOT).  Then for each $n\in \mathbb{N}$, we have 
$u_n\cdot x_0=x_0$, whence $(u_n-1)x_0=-(u_n-1)z$. In particular, 
$$\sup_{n\in \mathbb{N}}\|(u_n-1)z\|_2=\sup_{n\in \mathbb{N}}\|(u_n-1)x_0\|_2\le 2\|x_0\|_2<\infty.$$
However, this is impossible by Lemma \ref{lem: not in L2} and  $(u-1)z\notin L^2(M,\tau)$. Therefore $G$ does not have property (FH). For each $n\in \mathbb{N}$, define 
 $p_n=\sum_{k=n+1}^{\infty}e_k$ and $u_n=p_n-p_n^{\perp}\ (n\in \mathbb{N})$, where $(e_k)_{k=1}^{\infty}$ is a sequence of mutually orthogonal projections in $M$ such that $\tau(e_k)=1\,(k\in \mathbb{N})$. Then $u_n\in \mathcal{U}_p(M,\tau)$, and $u_n\stackrel{n\to \infty}{\to}-1$ (SOT).  
Therefore, by $-1\notin \mathcal{U}_p(M,\tau)$, $\mathcal{U}_p(M,\tau)$ does not have property (FH).
\end{proof}
\begin{corollary}
Let $2<p<M$ and $M$ be a semifinite von Neumann algebra with separable predual equipped with a normal faithful semifinite trace $\tau$. Then $\mathcal{U}_p(M,\tau)$ has property {\rm{(FH)}}, if and only if $M$ is of finite type. 
\end{corollary}
\begin{proof}
If $M$ is of finite type, then $\mathcal{U}_p(M,\tau)=\mathcal{U}(M)$ is bounded by Lemma \ref{lem: bounded for vNa}. Therefore it has property (FH). Conversely, if $M$ is not of finite type, then we may write $M=M_1\oplus M_2$, where 
$M_1$ is a finite von Neumann algebra (possibly $\{0\}$) and $M_2$ is a properly infinite semifinite von Neumann algebra. If $\mathcal{U}_p(M,\tau)=\mathcal{U}(M_1,\tau_1)\times \mathcal{U}(M_2,\tau_2)$ had property (FH), then so would $\mathcal{U}_p(M_2,\tau_2)$, where $\tau_i=\tau|_{M_i}\,(i=1,2)$. But this contradicts Theorem \ref{thm no FH for big G}. Thus, $\mathcal{U}_p(M,\tau)$ does not have property (FH).  
\end{proof}
\begin{remark}
If $M$ is a II$_{\infty}$ factor with separable predual, then the fact that $\mathcal{U}_p(M,\tau)\ (1\le p\le 2)$ does not have property (T) can be seen from the work of Enomoto--Izumi \cite{EnomotoIzumi}. Indeed, for each $t>0$, let $(\pi_t,H_t, \xi_t)$ be the cyclic representation of $\mathcal{U}(M,\tau)_2$ associated with the character $\varphi_t(u)=e^{-t\|u-1\|_2^2}\ (u\in \mathcal{U}_p(M,\tau))$. Define $\pi=\bigoplus_{n=1}^{\infty}\pi_{\frac{1}{n}}$. Then one can easily show that the trivial representation $\pi_0$ is weakly contained in $\pi$, but because each $\pi_t$ is a II$_1$ factor representation \cite[Theorem 1.6]{EnomotoIzumi}, $\pi$ does not contain $\pi_0$. Therefore $\mathcal{U}_p(M,\tau)$ does not have property (T). \end{remark}
Next, we give the second proof of the absence of property (FH) for $\mathcal{U}_p(M,\tau)$. This may be considered to be an analogue of Theorem \ref{not FH new}. 
\begin{theorem}\label{thm: (FH) implies Ufin for p>2}
Let $2<p<\infty$, $M$ be a properly infinite semifinite von Neumann algebra with separable predual equipped with a normal faithful semifinite trace $\tau$. Let $G$ be a closed subgroup of $\mathcal{U}_p(M,\tau)$ with property {\rm{(FH)}}. Then there exists a $\tau$-finite projection $e$ in $M$ such that $e$ reduces all elements in $G$,
that $u=1$ on ${\rm ran}\,(e)^{\perp}$, and that
the map
\[
G\to\mathcal{U}(M_e),\ \ \ \ \ u\mapsto u|_{{\rm ran}\,(e)}
\] 
defines an isomorphism from $G$ onto a closed subgroup of $\mathcal{U}(eMe)$. 
In particular, $G$ is a Polish group of finite type.
\end{theorem}
In the rest of this section, we fix $(M,\tau)$ and $G$ satisfying the hypothesis of Theorem \ref{thm: (FH) implies Ufin for p>2}. 
We represent $M$ on $L^2(M,\tau)$ by the standard representation given by $\tau$. Consider a map $\Psi_G\colon L^2(M,\tau)\to \prod_{u\in G}L^2(M,\tau)$ given by $$\Psi_G(x)=((u-1)x)_{u\in G},\,\,\,x\in L^2(M,\tau).$$
Also fix $2<q<\infty$ such that $\tfrac{1}{p}+\tfrac{1}{q}=\tfrac{1}{2}$. 

\begin{lemma}\label{lem: e comm with G}
${\rm{ker}}(\Psi_G)$ is a closed subspace of  $L^2(M,\tau)$. Let $e$ be the projection of $L^2(M,\tau)$ onto ${\rm{ker}}(\Psi_G)^{\perp}$. Then $e\in M$ and $ue=eu$ for every $u\in G$. 
\end{lemma}
\begin{proof}
It is clear that ${\rm{ker}}(\Psi_G)$ is a closed subspace of  $L^2(M,\tau)$. If $v\in M'$, then for every $u\in G$ and $x\in {\rm{ker}}(\Psi_G)$, $(u-1)vx=v(u-1)x=0$. Thus ${\rm{ker}}(\Psi_G)$ is invariant under $M'$, whence $e\in M$ holds. Moreover, if $u,v\in G$ and $x\in {\rm{ker}}(\Psi_G)$, then 
$(v-1)ux=(vu-1)x-(u-1)x=0$, whence $ux\in {\rm{ker}}(\Psi_G)$. Thus ${\rm{ker}}(\Psi_G)$ is invariant under $G$. This shows that $eu=ue$ for every $u\in G$. 
\end{proof}
Let $X=L^q(M,\tau)\cap \overline{M\cap {\rm{ker}}(\Psi_G)^{\perp}}^{\|\cdot\|_q}$ (use $M\cap L^2(M,\tau)\subset L^q(M,\tau))$. 
\begin{lemma}\label{lem: G action restricts}
If $z\in X$, $u\in G$ and $x\in {\rm{ker}}(\Psi_G)^{\perp}$, then $ux+(u-1)z\in {\rm{ker}}(\Psi_G)^{\perp}$ holds. 
\end{lemma}
\begin{proof}
Let $(z_n)_{n=1}^{\infty}$ be a sequence in $M\cap {\rm{ker}}(\Psi_G)^{\perp}$ such that $\|z_n-z\|_q\stackrel{n\to \infty}{\to}0$. Then for each $u\in G$ and $n\in \mathbb{N}$, 
$ux+(u-1)z_n\in {\rm{ker}}(\Psi_G)^{\perp}$ by Lemma \ref{lem: e comm with G}. Moreover, by the generalized noncommutative H\"older's inequality, we see that 
$\|(u-1)(z-z_n)\|_2\le \|u-1\|_p\|z-z_n\|_q\stackrel{n\to \infty}{\to}0.$ 
Thus $ux+(u-1)z\in {\rm{ker}}(\Psi_G)^{\perp}$ holds.  
\end{proof}
By Lemma \ref{lem: G action restricts}, for each $z\in X$, one can define a continuous affine isometric action $\alpha_z$ of $G$ on ${\rm{ker}}(\Psi_G)^{\perp}$ (regarded as a real Hilbert space by taking the real part of the inner product) by 
\[\alpha_z(u)x=ux+(u-1)z,\,u\in G,\,x\in {\rm{ker}}(\Psi_G)^{\perp}.\]
Because $G$ is assumed to have property (FH), there exists an $\alpha_z$-fixed point in ${\rm{ker}}(\Psi_G)^{\perp}$. Let $x,y\in {\rm{ker}}(\Psi_G)^{\perp}$ be two $\alpha_z$-fixed points. 
Then for every $u\in G$, $(u-1)x=(u-1)y$, whence $x-y$ belongs to ${\rm{ker}}(\Psi_G)$, which implies that $x=y$. Thus, there exists a unique $\alpha_z$-fixed point, which we denote by $f(z)\in {\rm{ker}}(\Psi_G)^{\perp}$.

\begin{lemma}\label{lem: fixed pt map is continuous} The map $f\colon X\ni z\mapsto f(z)\in {\rm{ker}}(\Psi_G)^{\perp}$ is an $L^q$-$L^2$ continuous linear map. 
\end{lemma}
\begin{proof} By the uniqueness of the fixed point, it follows that $f$ is a linear map. 
Because $X$ (resp. ${\rm{ker}}(\Psi_G)^{\perp}$) is a closed subspace of the Banach space $L^q(M,\tau)$ (resp. $L^2(M,\tau)$), thanks to the closed graph theorem, we only have to prove that $f$ is closed. Suppose $(z_n)_{n=1}^{\infty}$ is a sequence in $X$ converging to $z\in X$ and that $\|f(z_n)-y\|_2\stackrel{n\to \infty}{\to}0$ for some $y\in {\rm{ker}}(\Psi_G)^{\perp}$. Then for each $u\in G$, we have 
\[\|(u-1)(y-f(z_n))\|_2\le 2\|y-f(z_n)\|_2\stackrel{n\to \infty}{\to}0,\]
while $\|(u-1)(z_n-z)\|_2\le \|u-1\|_p\|z_n-z\|_q\stackrel{n\to \infty}{\to}0$. Thus, it follows that in the $L^2$-topology, we have
\[(u-1)y=\lim_{n\to \infty}(u-1)f(z_n)=\lim_{n\to \infty}-(u-1)z_n=-(u-1)z.\]
This shows that $\alpha_z(u)y=y$. Then by the uniqueness of the $\alpha_z$-fixed point, $y=f(z)$ holds. This finishes the proof. 
\end{proof}
\begin{lemma}\label{lem: tau e finite}
$e\in M$ is a $\tau$-finite projection. 
\end{lemma}
\begin{proof}
Assume by contradiction that $\tau(e)=\infty$. 
As in the proof of Proposition \ref{prop: unbounded L2}, we write $M=\bigoplus_{j\in J}Q_j\oplus Q_0$ as the direct sum of type I$_{\infty}$ factors and a diffuse part, and for each $j\in J$, define $c_j=\tau(q_j)$, where $q_j$ is a minimal projection and $c_0=1$. Write $e=(e_j)_{j\in J_0}$. We treat the case $J\neq \emptyset$ and $Q_0\neq \{0\}$. 
By assumption, $\sum_{j\in J_0}\tau(e_j)=\infty$ holds.\\ \\
\textbf{Case 1.} There exists $j\in J_0$ such that $\tau(e_j)=\infty$.\\
Then by using either the diffuseness of $Q_0$ or the fact that $Q_j\,(j\in J)$ is a type I$_{\infty}$ factor, there exists a sequence $(e_{j,n})_{n=1}^{\infty}$ of mutually orthogonal projections in $Q_j$ such that $\tau(e_{j,n})=c_j$ for every $n\in \mathbb{N}$ and $e_j=\sum_{n=1}^{\infty}e_{j,n}$. 
Define 
\[z_n:=\sum_{k=1}^nk^{-\frac{1}{2}}e_{j,k},\,n\in \mathbb{N}.\]
Then for each $n\in \mathbb{N}$, $z_n\in M\cap L^2(M,\tau)\subset L^q(M,\tau)$ holds. 
Moreover, because $e$ is the projection onto ${\rm{ker}}(\Psi_G)^{\perp}$ and $e_{j,n}\le e_n\le e$, we see that (we regard $e_{j,n}\in L^2(M,\tau)$) $z_n\in M\cap {\rm{ker}}(\Psi_G)^{\perp}$, hence $z_n\in X$ holds. 
Thus, because $z_n\in {\rm{ker}}(\Psi_G)^{\perp}$, the uniqueness of the $\alpha_{z_n}$-fixed point implies that $f(z_n)=-z_n$. On the other hand, by $q>2$, for each $n\in \mathbb{N}$, we have  
$$\|z_n\|_q^q=\sum_{k=1}^{n}k^{-\frac{q}{2}}c_j\le \sum_{k=1}^{\infty}k^{-\frac{q}{2}}c_j\,(<\infty),$$
whence $\sup_{n\in \mathbb{N}}\|z_n\|_q<\infty.$ On the other hand, 
\[\|f(z_n)\|_2^2=\|z_n\|_2^2=\sum_{k=1}^nk^{-1}c_j\stackrel{n\to \infty}{\to}\infty.\]
This contradicts the continuity of $f$ by Lemma \ref{lem: fixed pt map is continuous}.  
Therefore, $\tau(e)<\infty$ holds.\\ \\
\textbf{Case 2.} $\tau(e_j)<\infty$ for all $j\in J_0$.\\
Discarding all $j\in J_0$ for which $\tau(e_j)=0$, we may assume that $\tau(e_j)\neq 0$ for every $j\in J_0$. In this case, for each $j\in J$, there exists $d_j\in \mathbb{N}$ such that $e_{j}=\sum_{p=1}^{d_j}e_{j,p}$ holds, where $(e_{j,p})_{p=1}^{d_j}$ is a family of mutually orthogonal minimal projections in $Q_j$.\\
\textbf{Case 2-1.} $\sum_{j:\,c_j<1}\tau(e_j)<\infty$.\\
In this case, $\sum_{j:\,c_j\ge 1}\tau(e_j)=\infty$ holds. 
Let $j_1<j_2<\dots$ be the enumeration of the set $\tilde{J}=\{j\in J\mid c_j\ge 1\}$. Define $m(1,p)=p\,(1\le p\le d_{j_1})$ and $m(k,p)=m(k-1,d_{j_{k-1}})+p\,(1\le p\le d_{j_k})$ for $k\ge 2$, and set 
\[z_n=\sum_{k=1}^n\sum_{p=1}^{d_{j_k}}m(k,p)^{-\frac{1}{2}}c_{j_k}^{-\frac{1}{q}}e_{j,p}.\]
Then $z_n\in X\,(n\in \mathbb{N})$ as in Case 1, and  
\eqa{
\|z_n\|_q^q&=\sum_{k=1}^n\sum_{p=1}^{d_{j_k}}m(k,p)^{-\frac{q}{2}}c_{j_k}^{-1}c_{j_k}\\
&=\sum_{k=1}^{m(n,d_{j_n})}k^{-\frac{2}{q}}\le \sum_{k=1}^{\infty}k^{-\frac{2}{q}}<\infty,
}
whence $\sup_{n\in \mathbb{N}}\|z_n\|_q<\infty$. On the other hand, 
\eqa{
\|z_n\|_2^2&=\sum_{k=1}^n\sum_{p=1}^{d_{j_k}}m(k,p)^{-1}c_{j_k}^{-\frac{2}{q}}c_{j_k}\\
&=\sum_{k=1}^{m(n,d_{j_n})}k^{-1}c_{j_k}^{1-\frac{2}{q}}
\ge \sum_{k=1}^nk^{-1}\stackrel{n\to \infty}{\to}\infty
}
by $q>2$ and $c_{j_k}\ge 1\,(k\in \mathbb{N})$. Thus, as in Case 1, we get a contradiction.\\
\textbf{Case 2-2.}$\sum_{j:\,c_j<1}\tau(e_j)=\infty$. 
Let $j_1<j_2<\dots$ be the enumeration of the set $\tilde{J}=\{j\in J\mid c_j<1\}$. 
We may find a sequence $(e_n)_{n=1}^{\infty}$ of mutually orthogonal projections in $M$ such that for each $n\in \mathbb{N}$, $e_n$ is of the form $\sum_{j\in J_n}f_{j,n}$, 
where $J_n\subset \tilde{J}$ and $f_{j,n}\le e_j$ (it is possible that $J_n\cap J_m\neq \emptyset$ for some $n\neq m$) and that $1\le \tau(e_n)\le 2$ for every $n\in \mathbb{N}$. 
Then define $z_n=\sum_{k=1}^nk^{-\frac{1}{2}}e_k\in X$. 
By $1\le \tau(e_n)\le 2$, we have 
\[\|z_n\|_q^q=\sum_{k=1}^{n}k^{-\frac{q}{2}}\tau(e_k)\le 2\sum_{k=1}^{\infty}k^{-\frac{q}{2}}<\infty,\]
and 
\[\|z_n\|_2^2=\sum_{k=1}^nk^{-1}\tau(e_k)\ge \sum_{k=1}^nk^{-1}\stackrel{n\to \infty}{\to}\infty\]
then as in Case 1, we get a contradiction. 
\end{proof}
\begin{proof}[Proof of Theorem \ref{thm: (FH) implies Ufin for p>2}]
By Lemma \ref{lem: e comm with G} and Lemma \ref{lem: tau e finite}, $e$ is a $\tau$-finite projection in $M$ such that the map $G\ni u\mapsto u|_{\rm{ran}}(e)\in \mathcal{U}(eMe)$ defines a topological group isomorphism of $G$ onto a subgroup of $\mathcal{U}(eMe)$(cf.  Lemma \ref{lem: p norms agree with SOT}). Since $G$ and $\mathcal{U}(eMe)$ are Polish, the image of this isomorphism is closed. 
\end{proof}

\begin{corollary}\label{cor: closed subgroup of U_pl2 is compact p>2}
Let $2<p<\infty$, and let $G$ be a closed subgroup of $\mathcal{U}_p(\ell^2)$ with property {\rm{(FH)}}.
Then $G$ is a compact Lie group.
\end{corollary}
\begin{proof}
By Theorem \ref{thm: (FH) implies Ufin for p>2}, the proof follows in exactly the same way as Corollary \ref{cor: closed subgroup of U_pl2 is compact}.
\end{proof}
\subsection{Property (T)} 
We have seen that there are many examples of unital separable C$^*$-algebras whose unitary group with the norm topology are bounded as Polish groups. Consequently, such groups have property (FH). It is interesting to know whether they also have property (T). Pestov \cite[Example 6.6]{Pestov18} showed that the Fredholm unitary group $\mathcal{U}_{\infty}(\ell^2)$ does not have property (T), while it has property (OB) by Atkin's result \cite{Atkin89,Atkin91}. 
The notion of property (T) for C$^*$-algebras has been introduced by Bekka \cite{Bekka06} and recently Bekka and Ng \cite{BekkaNg19} have shown that if $G$ is a locally compact group, then it has property (T) if and only if the full group C$^*$-algebra $C^*(G)$ has property (T), and if moreover $G$ is an [IN]-group, then this condition is also equivalent to the strong property (T) (in the sense of Ng \cite{Ng14}) for the reduced group C$^*$-algebra $C_r^*(G)$. Still, the relationship between the property (T) for the unitary group and the C$^*$-algebra itself is not well understood. 
We show here that for many $A$, its unitary group fails to have property (T). This is an immediate consequence of the notion of the property Gamma of Murray and von Neumann. 
Recall that a type II$_1$ factor $M$ with normal faithful tracial state $\tau$ has {\it property Gamma}, if for each $x_1,\dots, x_n\in M$ and $\varepsilon>0$, there exists $u\in \mathcal{U}(M)$ such that $\|ux_i-x_iu\|_2<\varepsilon$ and $\tau(u)=0$.  
\begin{theorem}\label{thm: Gamma negates (T)}
Let $A$ be a unital separable {\rm{C}}$^*$-algebra admitting a $*$-representation $\pi$ such that $M=\pi(A)''$ is a type  ${\rm{II}}_1$ factor with separable predual having property Gamma. Then $\mathcal{U}(A)_u$ does not have property {\rm{(T)}}. In particular, this is the case if $A$ is a unital separable simple infinte-dimensional nuclear {\rm{C}}$^*$-algebra with a tracial state.  
\end{theorem}
\begin{proof}
Let $\tau$ be the unique faithful normal tracial state on $M=\pi(A)''$ and we represent $M$ on $L^2(M,\tau)$ by the GNS representation associated with $\tau$. We write $\hat{a}\ (a\in M)$ when we view $a$ as an element in $L^2(M,\tau)$. Recall that any $\varphi\in N_*^+$ is represented as $\varphi=\nai{\,\cdot\,\xi_{\varphi}}{\xi_{\varphi}}$ for a unique element $\xi_{\varphi}\in L^2(M,\tau)_+$. $L^2(M,\tau)$ is equipped with the canonical $M$-$M$ bimodule structure given by $x\cdot \xi\cdot y=xJy^*J\xi\ (x,y\in M,\,\xi\in L^2(M,\tau)$ and $J\hat{x}=\widehat{x^*}\ (x\in M)$. Let $H=L^2(M,\tau)\ominus \mathbb{C}\hat{1}$. Define a unitary representation $\rho\colon \mathcal{U}(A)_u\to \mathcal{U}(H)_s$ by 
\[\rho(u)\xi=\pi(u)\xi \pi(u)^*,\ \ \xi\in H,\,u\in \mathcal{U}(A).\]
Clearly $\rho$ is strongly continuous. 
Let $Q$ be a nonempty compact subset of $\mathcal{U}(A)_u$ and $\varepsilon>0$. By compactness, there exist $u_1,\dots u_m\in Q$ such that $Q\subset \bigcup_{j=1}^mB(u_j;\tfrac{\varepsilon}{3})$, where $B(u;r)=\{v\in \mathcal{U}(A)\mid \|u-v\|<r\}$. Let $u\in Q$. 
Because $M$ has property Gamma, there exists $v\in \mathcal{U}(M)$ such that $\|v\pi(u_j)-\pi(u_j)v\|_2<\varepsilon\,(1\le j\le m)$ and $\tau(v)=0$. Then $\xi=\hat{v}\in H$ is a unit vector satisfying $\|\rho(u_j)\xi-\xi\|_2<\varepsilon\ (1\le j\le m)$. 
Let $u\in Q$ and choose $1\le j\le m$ such that $\|u-u_j\|<\tfrac{\varepsilon}{3}$. Then 
\eqa{
\|\rho(u)\xi-\xi\|&\le \|\rho(u)\xi-\rho(u_j)\xi\|+\|\rho(u_j)\xi-\xi\|\\
&\le \|\pi(u-u_j)\xi \pi(u)^*\|+\|\pi(u_j)\xi \pi(u-u_j)^*\|+\tfrac{\varepsilon}{3}\\
&<\varepsilon.
}
This shows that $\rho$ almost has invariant vectors. Assume that $\eta\in H$ is a $\rho$-invariant vector. Let $\eta=v|\eta|$ be the polar decomposition of $\eta$ regarded as a closed and densely defined operator on $L^2(M,\tau)$ affiliated with $M$. Then $v\in M$ and $|\eta|\in L^2(M,\tau)_+$. By the $\rho$-invariance, we have $\pi(u)\eta=\eta \pi(u)$ for every $u\in \mathcal{U}(A)$. Since $\pi(A)'=M'$, this implies that for every $u\in \mathcal{U}(M)$, $u\eta u^*=\eta$ holds. Therefore 
\[(uvu^*)(u|\eta|u^*)=v|\eta|,\ \ \ u\in \mathcal{U}(M).\]
By the uniqueness of the polar decomposition, it follows that $v\in M'=JMJ$ and $u|\eta|u^*=|\eta|\ (u\in \mathcal{U}(M))$. Write $v=Jw^*J$ for $w\in \mathcal{U}(M)$. By the latter identity, the functional $\nai{\,\cdot\,|\eta|}{|\eta|}$ is tracial on $M$. Therefore by the uniqueness of the representing vector in the positive cone, there exists a nonnegative real $\lambda$ such that $|\eta|=\lambda \hat{1}$. Then $\eta=Jw^*J\lambda \hat{1}=\lambda \hat{w}$. 
But for each $u\in \mathcal{U}(M)$, the condition $u\eta u^*=\eta$ implies $\lambda \widehat{uwu^*}=\lambda \hat{w}$. Thus $w\in \mathcal{Z}(M)=\mathbb{C}1$. This shows that $\eta\in \mathbb{C}1$. But $\eta$ is orthogonal to $\hat{1}$ by hypothesis, whence $\eta=0$. This shows that $\rho$ does not admit a nonzero invariant vector. This shows that $\mathcal{U}(A)_u$ does not have property (T). The last assertion is then immediate, because if $\tau$ is an extremal point in the tracial state space of $A$, then $\tau$ is faithful (by the simplicity of $A$) and its associated GNS representation $\pi_{\tau}$ generates a hyperfinite type II$_1$ factor, which has property Gamma. 
\end{proof}
\begin{remark}
The same argument shows that $\mathcal{U}(M)$ as a discrete group does not have property (T) if $M$ is a type II$_1$ factor with separable predual and with property Gamma. This generalizes Pestov's result \cite[Example 5.5]{Pestov18} for $\mathcal{U}(R)$, where $R$ is the hyperfinite type II$_1$ factor. It is also immediate by using Connes--Jones' notion \cite{ConnesJones85} of property (T) for von Neumann algebras, that for a type II$_1$ factor $M$ with separable predual, if $\mathcal{U}(M)$ as a discrete group has property (T), then $M$ has property (T). The converse implication is however unclear.
\if0\\
\textcolor{blue}{
Assume $M$ does not have property (T). Then for each $\varepsilon>0$ and for each finite subset $\mathcal{F}\subset \mathcal{U}(M)$ there exists an $M-M$ bimodule $H=H(\mathcal{F},\varepsilon)$ and a unit vector $\xi\in H$ such that $\|u\xi-\xi u\|<\varepsilon$ for every $u\in \mathcal{F}$ but there is no nonzero $\eta\in H$ such that $x\eta=\eta x\, (x\in M)$ holds. 
Consider the unitary representation of $\mathcal{U}(M)$ on $\bigoplus_{\mathcal{F},\varepsilon}H_{\mathcal{F},\varepsilon}$
given by 
$\pi(u)(\xi_i)_{i\in I}=(u\xi_i u^*)_{i\in I}$ where $I=\{(\mathcal{F},\varepsilon)\mid \mathcal{F}\stackrel{finite}{\subset}\mathcal{U}(M),\varepsilon>0\}$. 
Then $\pi$ almost has invariant vectors, but there is no nonzero $\pi$-invariant vector. 
``The converse implication, namely $M$ has (T) $\Rightarrow \mathcal{U}(M)$ has (T) as a discrete group" is unclear because for general unitary representation $\pi$ of $\mathcal{U}(M)$ we do not know how to construct an $M-M$ bimodule.}
\fi 
\end{remark}
\begin{remark}\label{rem: not (T)}
As we mentioned in the beginning of this section, Pestov \cite[Example 6.6]{Pestov18} showed that the Fredholm unitary group $\mathcal{U}_{\infty}(\ell^2)$ does not have property (T). His proof is based on his result \cite[Theorem 6.3]{Pestov18} that a topological group which is both NUR and amenable must be maximally almost periodic, and the fact that $\mathcal{U}_{\infty}(\ell^2)$ is amenable (in fact, it is extremely amenable by Gromov--Milman's result \cite{GromovMilman}). 
We remark that the same type of argument as the one in Theorem \ref{thm: Gamma negates (T)} provides another elementary proof of the absence of property (T) for $\mathcal{U}_{\infty}(\ell^2)$. This is likely to be known or implicit in the literature, but we include the proof for the reader's convenience.   
Consider the strongly continuous unitary representation $\pi\colon \mathcal{U}_{\infty}(\ell^2)\to \mathcal{U}(S_2(\ell^2))$ given by 
\[\pi(u)x=uxu^*,\ \ u\in \mathcal{U}_{\infty}(\ell^2),\,x\in S_2(\ell^2).\]
Then $\pi$ almost has invariant vectors (e.g. if $(\xi_n)_{n=1}^{\infty}$ is an orthonormal basis for $\ell^2$, then the rank one projection $p_n$ onto $\mathbb{C}\xi_n\,(n\in \mathbb{N})$ satisfies $\|\pi(u)p_n-p_n\|_2\stackrel{n\to \infty}{\to}0\,(u\in \mathcal{U}_{\infty}(\ell^2))$, but it is straightforward to see that there is no nonzero $\pi$-invariant vector. 
\end{remark}
 
Let $M_{2^{\infty}}$ be the UHF algebra of type $2^{\infty}$ (the CAR algebra). Then $\mathcal{U}(M_{2^{\infty}})_u$ is bounded by Proposition \ref{prop: boundedness of U(A) for AF}, hence it has property (FH). We know that $\mathcal{U}(M_{2^{\infty}})_u$ does not have property (T) by Theorem \ref{thm: Gamma negates (T)}. The absence of the property (T) can also be seen from the existence of certain type III factor representations. The reason for presenting an alternative proof is our hope that this method could be used to show the absence of property (T) for C$^*$-algebras without tracial states. We will not pursue this approach further in this paper. 
\begin{proposition}\label{prop: U(CAR) does not have (T)}
The group $\mathcal{U}(M_{2^{\infty}})_u$ does not have property {\rm{(T)}}.
\end{proposition}  
For the proof, we recall the construction of the Powers factors $R_{\lambda}\ (0<\lambda<1)$. 
For $0<\lambda<1$, let $h_{\lambda}=\tfrac{1}{1+\lambda}\mattwo{1}{0}{0}{\lambda}$ and $\psi_{\lambda}=\text{Tr}(h_{\lambda}\ \cdot\ )$, which is a normal faithful state on $M_2(\mathbb{C})$.  Let $\varphi_{\lambda}=\bigotimes_{n=1}^{\infty}\psi_{\lambda}$ be the infinite tensor product state on $M_{2^{\infty}}$ and let $(H_{\lambda},\pi_{\varphi_{\lambda}},\xi_{\varphi_{\lambda}})$ be the associated GNS representation of $M_{2^{\infty}}$. Then $R_{\lambda}=\pi_{\varphi_{\lambda}}(M_{2^{\infty}})''$ is called the Powers factor (of type III$_{\lambda})$. Let $J_{\psi_{\lambda}}$ (resp. $\Delta_{\psi_{\lambda}}$) be the modular conjugation operator (resp. the modular operator) associated with $\psi_{\lambda}$ defined on its GNS representation. Similarly we let $J_{\varphi_{\lambda}}$ and $\Delta_{\varphi_{\lambda}}$ be the corresponding modular objects associated with $\varphi_{\lambda}$. Then 
\[J_{\varphi_{\lambda}}=\bigotimes_{n=1}^{\infty}J_{\psi_{\lambda}},\ \ \ \Delta_{\varphi_{\lambda}}=\bigotimes_{n=1}^{\infty}\Delta_{\psi_{\lambda}}.\]
We use the notation $\Lambda_{\psi_{\lambda}}\colon M_2(\mathbb{C})\to H_{\psi_{\lambda}}$ and $\Lambda_{\varphi_{\lambda}}\colon M_{2^{\infty}}\to H_{\varphi_{\lambda}}$ for the canonical dense embedding. For each $n\in \mathbb{N}$, let $\lambda_n=1-\tfrac{1}{n+1}$. Set 
\[\varphi_n=\varphi_{\lambda_n},\ \ \psi_n=\psi_{\lambda_n},\ H_n=H_{\varphi_n},\ \ \pi_n=\pi_{\varphi_n}.\] 
Then for each $a,b,c,d\in \mathbb{C}$, the following formulas hold:
\begin{align*}
\Delta_{\psi_{\lambda}}^{\frac{1}{2}}\Lambda_{\psi_{\lambda}}\left (\mattwo{a}{b}{c}{d}\right )&=\Lambda_{\psi_{\lambda}}\left (\mattwo{a}{\lambda^{-\frac{1}{2}}b}{\lambda^{\frac{1}{2}} c}{d}\right ),\\
J_{\psi_{\lambda}}\Lambda_{\psi_{\lambda}}\left (\mattwo{a}{b}{c}{d}\right )&=\Lambda_{\psi_{\lambda}}\left (\mattwo{\overline{a}}{\lambda^{-\frac{1}{2}}\overline{c}}{\lambda^{\frac{1}{2}}\overline{b}}{\overline{d}}\right ).
\end{align*}
\begin{proof}[Proof of Proposition \ref{prop: U(CAR) does not have (T)}] Let $A=M_{2^{\infty}}\supset A_0=\bigotimes_{\rm{alg}}M_2(\mathbb{C})$. Let $\lambda_n=1-\tfrac{1}{n+1}\ (n\in \mathbb{N})$ and consider the GNS representation $\pi_n\colon A\to \mathbb{B}(H_n)$ of the Powers state $\varphi_{\lambda_n}$. Then $M_n=\pi_n(A)''$ is the Powers factor of type III$_{\lambda_n}$. Let $\xi_n\in H_n$ be the GNS vector of $\varphi_{\lambda_n}$. 
Let $\Delta_n=\Delta_{\varphi_{\lambda_n}},\,J_n=J_{\varphi_{\lambda_n}},\,\,H=\bigoplus_{n=1}^{\infty}H_n$, $\pi=\bigoplus_{n=1}^{\infty}\pi_n\colon A\to \mathbb{B}(H)$ and $\tilde{\xi}_n=(0,\dots ,0,\xi_n,0,\dots)\in H$. Each $H_n$ is endowed with the standard $M_n-M_n$ bimodule structure given by 
 \[a\cdot \xi\cdot b=aJ_nb^*J_n\xi,\ \ \ a,b\in M_n,\ \xi_n\in H_n.\]
 Define a unitary representation $\rho\colon \mathcal{U}(A)\to \mathcal{U}(H)$ by 
 \begin{equation}
 \rho(u)(\eta_n)_{n=1}^{\infty}=(\pi_n(u)\eta_n\pi_n(u)^*)_{n=1}^{\infty},\ \ \ u\in \mathcal{U}(A),\ (\eta_n)_{n=1}^{\infty}\in H.
 \end{equation}
 We show that $\rho$ almost has invariant vectors. To this end, we first show that 
 \begin{equation}
 \lim_{n\to \infty}\|\pi_n(a)\xi_n-\xi_n\pi_n(a)\|=0,\ \ \ \ a\in A_0.\label{eq: almost inv A_0}
 \end{equation}
 Since $A_0$ is spanned by elementary tensors, it suffices to show (\ref{eq: almost inv A_0}) for elements of the form $a=a_1\otimes \cdots a_k\otimes 1\cdots \in A_0$, where $a_1,\dots, a_k\in M_2(\mathbb{C})$ and $k\in \mathbb{N}$. By construction, we may identify $H_n$ as the infinite tensor product Hilbert space $\bigotimes_{m=1}^{\infty}(M_2(\mathbb{C}),\xi_{\psi_{\lambda_n}})$, and $\xi_n=\xi_{\psi_{\lambda_n}}^{\otimes k}\otimes \bigotimes_{j=k+1}^{\infty}\xi_{\psi_{\lambda_n}}$. We have 
 \eqa{
\pi_n(a)\xi_n&=\bigotimes_{j=1}^k\Lambda_{\psi_{\lambda_n}}(a_j)\otimes \bigotimes_{j=k+1}^{\infty}\xi_{\psi_{\lambda_n}},\\
\xi_n\pi_n(a)
&=\bigotimes_{j=1}^k\Delta_{\psi_{\lambda_n}}^{\frac{1}{2}}\pi_n(a_j)\xi_{\psi_{\lambda_n}}\otimes \bigotimes_{j=k+1}^{\infty}\xi_{\psi_{\lambda_n}}.
}
Thus 
\eqa{
\|\pi_n(a)\xi_n-\xi_n\pi_n(a)\|=
\left \|\bigotimes_{j=1}^k\pi_{\psi_{\lambda_n}}(a_j)\xi_{\psi_{\lambda_n}}-
\bigotimes_{j=1}^k\Delta_{\psi_{\lambda_n}}^{\frac{1}{2}}\pi_n(a_j)\xi_{\psi_{\lambda_n}}\right \|.
}
Because the map $H^n\ni (\eta_1,\dots, \eta_n)\mapsto \eta_1\otimes \cdots \otimes \eta_n\in H^{\otimes n}$ is continuous, in order to prove (\ref{eq: almost inv A_0}) for $a=a_1\otimes \cdots \otimes a_k$, it suffices to show that $\|\pi_{\psi_{\lambda_n}}(b)\xi_{\psi_{\lambda_n}}-\Delta_{\psi_{\lambda_n}}^{\frac{1}{2}}\pi_{\psi_{\lambda_n}}(b)\xi_{\psi_{\lambda_n}}\|\stackrel{n\to \infty}{\to}0$ for all $b\in M_2(\mathbb{C})$. Let $b=\mattwo{p}{q}{r}{s}\in M_2(\mathbb{C})$. Then by $\lambda_n\stackrel{n\to \infty}{\to}1$, we have 
\[\|\pi_{\psi_{\lambda_n}}(b)\xi_{\psi_{\lambda_n}}-\Delta_{\psi_{\lambda_n}}^{\frac{1}{2}}\pi_{\psi_{\lambda_n}}(b)\xi_{\psi_{\lambda_n}}\|^2
=\frac{1}{1+\lambda_n}\{(1-\lambda_n^{\frac{1}{2}})^2|r|^2+\lambda_n(1-\lambda_n^{-\frac{1}{2}})^2|q|^2\}
\stackrel{n\to \infty}{\to}0.\]
This shows (\ref{eq: almost inv A_0}). 
Let $K\subset \mathcal{U}(A)$ be a nonempty norm-compact subset and $\varepsilon>0$. 
Because $K$ is totally bounded, there exist $u_1,\dots, u_m\in K$ such that $K\subset \bigcup_{i=1}^mB(u_i; \tfrac{\varepsilon}{3})$, where $B(u;r)=\{v\in \mathcal{U}(A);\ \|v-u\|<r\}\ (u\in \mathcal{U}(A))$. Since $A_0$ is norm-dense  in $A$, $\mathcal{U}(A_0)$ is norm-dense in $\mathcal{U}(A)$ (well-known). For each $i=1,\dots m$, there exists $v_i\in \mathcal{U}(A_0)$ such that $\|u_i-v_i\|<\tfrac{\varepsilon}{3}$. Then 
$K\subset \bigcup_{i=1}^mB(v_i;\tfrac{2}{3}\varepsilon)$.  By (\ref{eq: almost inv A_0}), for each $i=1,\dots, m$, there exists $n_i\in \mathbb{N}$ such that $\|\pi_n(v_i)\xi_n-\xi_n\pi_n(v_i)\|=\|\pi_n(v_i)\xi_n\pi_n(v_i)^*-\xi_n\|<\tfrac{\varepsilon}{3}$ holds for all $n\ge n_i$. 
Let $n=\max_{1\le i\le m}n_i$. Then for each $i=1,\dots, m$, 

\eqa{
\|\rho(v_i)\tilde{\xi}_{n}-\tilde{\xi}_{n}\|&=\|\pi_{n}(v_i)\xi_n\pi_{n}(v_i)^*-\xi_n\|<\tfrac{\varepsilon}{3}.
}

If $u\in K$, then $\|u-v_i\|<\tfrac{2}{3}\varepsilon$ for some $1\le i\le m$. Thus 
\eqa{
\|\rho(u)\tilde{\xi}_{n}-\tilde{\xi}_{n}\|&\le \|(\rho(u)-\rho(v_i))\tilde{\xi}_{n}\|+\|\rho(v_i)\tilde{\xi}_{n}-\tilde{\xi}_{n}\|\\
&\le \|\pi_n(u-v_i)\tilde{\xi}_{n}\pi_{n}(u)^*\|+\|\pi_n(v_i)\tilde{\xi}_n(\pi_n(u^*-v_i^*)\|+\tfrac{\varepsilon}{3}\\
&\le 2\|u-v_i\|+\tfrac{\varepsilon}{3}<\varepsilon.
}
Since $u\in K$ is arbitrary, $\rho$ almost has invariant vectors.\\ \\
Next, we show that $\rho$ does not have a nonzero invariant vector. Assume by contradition that $\eta=(\eta_n)_{n=1}^{\infty}\in H$ is a nonzero $\rho$-invariant vector. Then there exists $n\in \mathbb{N}$ such that $\eta_n\neq 0$, and $\pi_n(u)\eta_n\pi_n(u)^*=\eta_n$ for all $u\in \mathcal{U}(A)$. We may assume that $\|\eta_n\|=1$. Let $\tau_n=\nai{\ \cdot\ \eta_n}{\eta_n}\in (M_n)_*^+$. 
Then for each $u,v\in \mathcal{U}(A)$, 
\eqa{
\tau_n(\pi_n(u)\pi_n(v))&=\nai{\pi_n(u)\pi_n(v)\eta_n}{\eta_n}=\nai{\pi_n(u)\eta_n\pi_n(v)}{\eta_n}\\&=\nai{\pi_n(u)\eta_n}{\eta_n\pi_n(v)^*}=\nai{\pi_n(u)\eta_n}{\pi_n(v)^*\eta_n}\\
&=\nai{\pi_n(v)\pi_n(u)\eta_n}{\eta_n}=\tau_n(\pi_n(v)\pi_n(u)).
}
Since $\text{span}(\mathcal{U}(A))=A$, we have $\tau_n(ab)=\tau_n(ba)$ for all $a,b\in \pi_n(A)$, whence for all $a,b\in M_n$ by the normality of $\tau_n$. This shows that $\tau_n$ is a tracial state. 
This contradicts the fact that $M_n$ is of type III$_{\lambda_n}$. Therefore $\rho$ does not have a nonzero invariant vector. Hence $\mathcal{U}(A)$ does not have property (T). 
\end{proof}
\section*{Appendix}
For the reader's convenience, we include proofs of the three well-known results we need in $\S$\ref{subsec: SUR}.   
\begin{lemma}\label{lem: cotypeC(X)} Let $X$ be an infinite compact Hausdorff space. Then then the Banach space $C(X)$ does not have nontrivial cotype. 
\end{lemma}
\begin{proof} We consider two cases separately.\\ 
\textbf{Case 1.} $X$ has a connected component $X_0$ consisting of more than two points, say $x,x'\in X_0\ (x\neq x')$, then since $X_0$ is a normal space, by Urysohn's Lemma, there exists a continuous function $f\colon X_0\to [0,1]$ satisfying $f(x)=0,\ f(x')=1$. Because $X_0$ is connected, this shows that $f(X_0)=[0,1]$. By Tietze's extension Theorem, $f$ can be extended to a continuous function on $X$, which we still denote by $f$. Then the dual map $f^*\colon C[0,1]\to C(X)$ is an isometric embedding. Since $C[0,1]$ does not have nontrivial cotype, $C(X)$ does not either.\\ \\
\textbf{Case 2.} $X$ is totally disconnected. Since $X$ is compact, there exist at most finitely many isolated points in $X$. Thus $X$ has a connected component $X_0$ which is totally disconnected and perfect. Since $X_0$ is compact Hausdorff, $X_0$ is therefore zero-dimensional, meaning that it has a neighborhood basis consisting of clopen subsets. Then because $X_0$ is infinite, for each $n\in \mathbb{N}$, there exists disjoint clopen sets $U_1,\dots,U_n\subset X_0$. Define $x_i=1_{U_i}\in C(X)\ (1\le i\le n)$. Let $(\varepsilon_n)_{n=1}^{\infty}$ be a Rademacher sequence and let $q\in [2,\infty)$. Then 
\[\left (\sum_{i=1}^n\|x_i\|^q\right )^{\frac{1}{q}}=n^{\frac{1}{q}},\ \ \mathbb{E}\left \|\sum_{i=1}^n\varepsilon_ix_i\right \|^{\frac{1}{q}}=1.\]
Since $n\in \mathbb{N}$ is arbitrary, this shows that $C(X)$ cannot have cotype $q$, which shows the lemma.        
\end{proof}

\begin{lemma}\label{lem: infinite sep subalg}
Let $A$ be an infinite-dimensional unital abelian {\rm{C}}$^*$-algebra. Then there exists a separable infinite-dimensional unital ${\rm{C}}^*$-subalgebra $B$ of $A$. 
\end{lemma}
\begin{proof}
Let $X$ be the spectrum of $A$ and identify $A=C(X)$. $X$ is then an infinite compact Hausdorff space. If $X$ is not totally disconnected, then as in the proof of Lemma \ref{lem: cotypeC(X)}, $C(X)$ contains an isometric copy of $C[0,1]$. If $X$ is totally disconnected, then because $X$ is compact, there are at most finitely many isolated points, so that there exists an infinite closed totally disconnected perfect subset $X_1$ of $X$. Then $X_1$ is zero-dimensional, and there exists a countable disjoint family $\{U_n\}_{n=1}^{\infty}$ of nonempty clopen subsets of $X_1$. Then the self-adjoint element $x=\sum_{n=1}^{\infty}2^{-n}1_{U_n}\in C(X)$ has infinite spectrum. Thus $B=C^*(x,1)$ is a separable infinite-dimensional C$^{*}$-subalgebra of $A$.   
\end{proof}
\begin{proposition}\label{prop: masa is infinitedim} 
Let $A$ be an infinite-dimensional unital {\rm{C}}$^{\ast}$-algebra and $B$ be a maximal abelian $*$-subalgebra of $A$. Then $B$ is infinite-dimensional. 
\end{proposition}
\begin{proof} This can be found e.g., in \cite[Exercise 4.6.12]{KadisonRingroseI}. We show the contrapositive. Assume that $B$ is finite-dimensional. Let $I=\{1,2,\dots,n\}$, $X$ be the spectrum of $B$, which is a finite set, say $\{x_1,\dots,x_n\}$. Let $p_i=1_{\{x_i\}}\in B$. Then $\{p_1,\dots,p_n\}$ is an orthogonal family of projections in $B$ and $\sum_{i=1}^np_i=1$. Since $B$ is maximal abelian, we have $p_iAp_i=\mathbb{C}p_i$ for every $1\le i\le n$.
In particular, there exists a state $\varphi_i$ on $A$ such that $p_ixp_i=\varphi_i(x)p_i$ for all $x\in A,\, i\in I$.  Define a relation $\sim$ on $I$ by 
$$i\sim j\Leftrightarrow p_jAp_i\neq \{0\}.$$
Then $\sim$ is an equivalence relation. Only the transitivity is not entirely obvious. 
If $i\sim j$ and $j\sim k$, then $p_jxp_i\neq 0$ and $p_kyp_j\neq 0$ for some $x,y\in A$.  Note that for $a\in A$, $$(p_jap_i)^*(p_jap_i)=p_ia^*p_jap_i=\varphi_i(a^*p_ja)p_i\neq 0,\,\  (p_jap_i)(p_jap_i)^*=\varphi_j(ap_ia^*)p_j.$$ 
Therefore $p_jap_i\neq 0$ if and only if $\varphi_i(a^*p_ja)\neq 0$, 
if and only if $\varphi_j(ap_ia^*)\neq 0$. Then $z=p_kyp_jxp_i\in p_kAp_i$ satisfies 
\eqa{
z^*z&=p_ix^*p_jy^*p_kyp_jxp_i=p_ix^*\varphi_j(y^*p_ky)p_jxp_i\\
&=\varphi_i(x^*p_jx)\varphi_j(y^*p_ky)p_i\neq 0
}
by $p_jxp_i\neq 0\neq p_kyp_j$. Thus $i\sim k$. Next we show\\ \\
\textbf{Claim.} If $p_jAp_i\neq \{0\}$, then $\dim(p_jAp_i)=1$.\\
Fix $y\in A$ such that $p_jyp_i\neq 0$. Then for every $x\in A$, we have 
\eqa{
&(p_jxp_iy^*p_j)yp_i=p_jx(p_iy^*p_jyp_i)\\
&\Leftrightarrow \varphi_j(xp_iy^*)p_jyp_i=\varphi_i(y^*p_jy)p_jxp_i,
}
whence (by $\varphi_i(y^*p_jy)\neq 0$)
\[p_jxp_i=\varphi_i(y^*p_jy)^{-1}\varphi_j(xp_iy^*)p_jyp_i\in \mathbb{C}p_jyp_i.\]
Let $I=I_1\sqcup \cdots \sqcup I_d$ be the decomposition of $I$ into $\sim$ equivalence classes. For each $1\le k\le d$ and $i,j\in I_k$, fix an element $x_{i,j}\in A$ with $p_jx_{i,j}p_i\neq 0$. Then $|p_jx_{i,j}p_i|=\varphi_i(x_{i,j}^*p_jx_{i,j})^{\frac{1}{2}}p_i$, so that 
$p_jx_{i,j}p_i=u_{ij}|p_ix_{i,j}p_i|$ gives the polar decomposition of $p_jx_{i,j}p_i$, where 
\[u_{i,j}=\varphi_i(x_{i,j}^*p_jx_{i,j})^{-\frac{1}{2}}p_jxp_i\in A.\]
Then $p_jAp_i=\mathbb{C}u_{i,j}$ if $i\sim j$ and $p_jAp_i=\{0\}$ otherwise. It is also straightforward to check that $\{u_{i,j};\ i,j\in I_s\}$ is a system of matrix units in $q_sAq_s$ with $q_s=\sum_{i\in I_s}p_i$.  
Therefore for every $x\in A$ we have 
\[x=\sum_{i,j\in I}p_jxp_i=\sum_{s=1}^d\sum_{i,j\in I_s}a_{i,j}(x)u_{i,j}\]
for some $a_{i,j}(x)\in \mathbb{C}$, and $A\cong \bigoplus_{s=1}^dM_{|I_s|}(\mathbb{C})$ is finite-dimensional.  
\end{proof}
\section*{Acknowledgments}
HA is supported by JSPS KAKENHI 16K17608 and Grant for Basic Science Research Projects from The Sumitomo Foundation. 
YM is supported by KAKENHI 26800055 and 26350231.

Hiroshi Ando\\
Department of Mathematics and Informatics, Chiba University, 1-33 Yayoi-cho, Inage, Chiba, 263-
8522, Japan\\
hiroando@math.s.chiba-u.ac.jp\\ \\
Yasumichi Matsuzawa\\
Department of Mathematics, Faculty of Education, Shinshu University, 6-Ro, Nishi-nagano, Nagano,
380-8544, Japan\\
myasu@shinshu-u.ac.jp
\end{document}